\documentclass[a4paper,10pt]{article}

\usepackage{cmap}
\usepackage[T1]{fontenc}
\usepackage[utf8]{inputenc}
\usepackage[english]{babel}
\usepackage{amsfonts,amssymb,amsthm}
\usepackage{xcolor}
\usepackage{amsmath}
\usepackage{a4wide}
\usepackage[affil-it]{authblk}

\usepackage{graphicx}
\usepackage{epstopdf}
\usepackage{secdot}
\usepackage{enumitem}

\numberwithin{equation}{section}

\usepackage[colorlinks=true,linktocpage,pdfpagelabels,
bookmarksnumbered,bookmarksopen]{hyperref}
\definecolor{ForestGreen}{rgb}{0.1,0.6,0.05}
\definecolor{EgyptBlue}{rgb}{0.063,0.1,0.6}
\hypersetup{
	colorlinks=true,
	linkcolor=EgyptBlue,         
	citecolor=ForestGreen,
	urlcolor=olive
}

\usepackage[hyperpageref]{backref}

\newtheorem{thm}{Theorem}[section]
\newtheorem{lemma}[thm]{Lemma}

\theoremstyle{definition}
\newtheorem{remark}[thm]{Remark}
\newtheorem{conjecture}[thm]{Conjecture}

\usepackage{marginnote}

\title{
	\vspace*{-1cm}
	Existence and multiplicity results for a class of semilinear elliptic equations
	\\ \medskip}

\author[1]{Vladimir Bobkov\thanks{E-mail: \texttt{bobkov@kma.zcu.cz}}}
\author[1]{Pavel Dr\'abek\thanks{E-mail: \texttt{pdrabek@kma.zcu.cz}}}
\affil[1]{{\small Department of Mathematics and NTIS, 
		Faculty of Applied Sciences, 
		University of West Bohemia, 
		Univerzitn\'i 8, 301 00 
		Plze\v{n}, Czech Republic}}
\author[2]{Jes\'us Hern\'andez\thanks{E-mail: \texttt{jesus.hernande@telefonica.net}}}
\affil[2]{{\small Instituto de Matem\'atica Interdisciplinar, 
		Facultad de Matem\'aticas,
		28040 Madrid, Spain
	}}

\date{}

\begin{document}
\maketitle
	
\begin{abstract}
	We study the existence and multiplicity of nonnegative solutions, as well as the behaviour of corresponding parameter-dependent branches, to the equation $-\Delta u = (1-u) u^m - \lambda u^n$ in a bounded domain $\Omega \subset \mathbb{R}^N$ endowed with the zero Dirichlet boundary data, where $0<m \leq 1$ and $n>0$. 
	When $\lambda > 0$, the obtained solutions can be seen as steady states of the corresponding reaction-diffusion equation describing a model of isothermal autocatalytic chemical reaction with termination. 
	In addition to the main new results, we formulate a few relevant conjectures.
	
	
\end{abstract}

\section{Introduction}

Let $\Omega$ be a smooth bounded domain in $\mathbb{R}^N$, $N \geq 1$, with the boundary $\partial \Omega$. 
We consider the semilinear elliptic problem
\begin{equation}\label{eq:D}
\left\{
\begin{aligned}
 -\Delta u &= (1-u) u^m - \lambda u^n &&\text{in}~ \Omega,\\
 u &\geq 0, ~ u \not\equiv 0 &&\text{in}~ \Omega,\\
 u &= 0 &&\text{on}~ \partial\Omega,
\end{aligned}
\right.
\end{equation}
where $\lambda$ is a real parameter and exponents $m, n$ satisfy
\begin{equation}\label{eq:mn}
0 < m \leq 1 
\quad \text{and} \quad 
n>0.
\end{equation}
Notice that the problem \eqref{eq:D} has a non-Lipschitz nonlinearity if $0<m<1$ or $0<n<1$.

The equation in the problem \eqref{eq:D} can be seen as a general dimension, steady state version of the reaction-diffusion equation
\begin{equation}\label{eq:LN}
u_t - u_{xx} = (1-u) u^m - \lambda u^n, \quad u \geq 0,
\end{equation}
which serves as a model of isothermal autocatalytic chemical reaction with termination provided $\lambda > 0$, see \textsc{Leach \& Needham} \cite[Chapters 6-9]{needham}.
The Cauchy and the initial-boundary value problems on $\mathbb{R}$ for the equation \eqref{eq:LN} were intensively studied and various results on the existence, qualitative properties, and asymptotic behaviour of the corresponding solutions were obtained, see, e.g., the works \cite{leachmccabe1,leachmccabe2,leachmccabe3} with references therein, and also \cite{CEF,DHI2,DK,dickstein} for related problems. 
In particular, it was observed that if $0<n<m<1$ and $0 < \lambda < \lambda_c$, where 
$$
\lambda_c = \left(\frac{n+1}{m+1}\right) \left(\frac{m+2}{m+1}\right)^{m-n} \frac{(m-n)^{m-n}}{(m-n+1)^{m-n+1}},
$$
then the equation \eqref{eq:LN} possesses a steady state solution with compact support in $\mathbb{R}$, see  \cite[Section 8.4.4]{needham}.
Thus, this solution satisfies also the problem \eqref{eq:D} in an appropriate bounded interval $\Omega \subset \mathbb{R}$. 

Hereinafter, with a slight abuse of notation, we say that $u$ is a \textit{compact support solution} to \eqref{eq:D} if either $u > 0$ in $\Omega$ and $\frac{\partial u}{\partial \nu} = 0$ on $\partial \Omega$, or the support of $u$ is a compact subset of $\Omega$.	
(In the former case, the solution is sometimes called \textit{flat}.)
That is, compact support solutions to \eqref{eq:D} are, in fact, solutions of the overdetermined problem with both Dirichlet and Neumann zero boundary data. 
Clearly, the presence of compact support solutions is a consequence of the non-Lipschitz nature of the nonlinearity of \eqref{eq:D} which results in the possible violation of the strong maximum principle (see Remark \ref{rem:smp} below) and also in the extinction in finite time phenomenon of the associated parabolic problem. 
For extensive information on compact support solutions we refer the interested reader to the books by \textsc{D\'iaz} \cite{D} and \textsc{Pucci \& Serrin} \cite{puser}.

Motivated by the above mentioned example of the one-dimensional compact support solution to \eqref{eq:D}, it is natural to rise the question about the existence of such solutions in higher dimensions. 
However, the available literature on even simpler problems with non-Lipschitz nonlinearities indicates that the higher-dimensional case is considerably more difficult than its one-dimensional counterpart.
The first natural approach is to search for compact support solutions to \eqref{eq:D} in the class of radially symmetric functions. This indeed can be done, see, e.g., \cite{FLS,GST} and Remarks \ref{rem:caseI:compact} and \ref{rem:caseV:compact} below. 
On the other hand, if one considers small perturbations of \eqref{eq:D} by nonradial weights, then radially symmetric arguments cannot be applied, while compact support solutions are still expected to exist. 
Thus, different approaches as, e.g., in \cite{DHI1,DHI2,DHI3} have to be employed. The existence theory for compact support solutions developed in these works heavily relies on variational considerations and properties of the solution set in the \textit{bounded} domain case.

In the present research, we study the existence and multiplicity questions for the problem \eqref{eq:D}, which not only provides the first step towards results on the existence of compact support solutions regardless the radial symmetry assumption, but is also interesting by itself since the obtained results appear to be considerably different from the entire space case $\Omega = \mathbb{R}^N$. 
In order to systematize possible choices of $m,n$ satisfying \eqref{eq:mn}, let us first fix $m \in (0,1)$ and vary $n>0$. 
We distinguish among the following four cases, see Figure \ref{fig:cases1}:
\begin{enumerate}[label={\rm \Roman*.~}, ref={\rm \Roman*}, itemindent=0.4cm]
	\item\label{I}   $0<n < m<1$,
	\item\label{II}  $0<m\leq n\leq 1$,
	\item\label{III} $0<m<1<n< m+1$,
	\item\label{IV}  $0<m<1<m+1\leq n$.
\end{enumerate}

\begin{figure}[ht]
	\centering
	\includegraphics[width=0.45\linewidth]{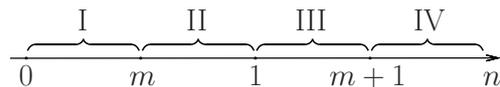}\\
	\caption{Cases I-IV for a fixed $m \in (0,1)$.}
	\label{fig:cases1}
\end{figure}

Fixing now $m=1$ and varying $n>0$, we allocate three additional cases, see Figure \ref{fig:cases2}:
\begin{enumerate}[label={\rm \Roman*.~}, ref={\rm \Roman*}, itemindent=0.4cm]\setcounter{enumi}{4}
	\item\label{VI}  	$0<n\leq m=1$, 
	\item\label{VII} 	$m=1<n \leq m+1=2$,
	\item\label{VIII}	$m+1=2< n$.
\end{enumerate}

\begin{figure}[ht]
	\centering
	\includegraphics[width=0.45\linewidth]{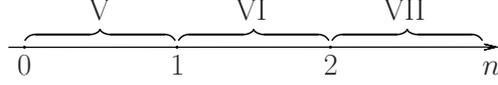}\\
	\caption{Cases V, VI, VII.}
	\label{fig:cases2}
\end{figure}

Moreover, for the sake of generality, we do not restrict the consideration only to the values $\lambda>0$, but also study the complementary case $\lambda \leq 0$.

The article has the following structure.
In Section \ref{section:prelim}, we provide some preliminary and auxiliary results.
Sections \ref{sec:I}-\ref{sec:IV} are devoted to the detailed study of the cases \ref{I}-\ref{IV}, and Sections \ref{sec:VI}-\ref{sec:VIII} cover the cases \ref{VI}-\ref{VIII}. 
Along the text, different arguments are used in order to prove the existence and multiplicity (or uniqueness) of positive solutions to \eqref{eq:D}. 
While the application of classical methods such as global minimization or the method of sub and supersolutions is relatively standard, the application of several variational methods, mostly in order to obtain multiplicity results, is more involved. 
Some of the stated assumptions are probably technical and could be eliminated. 
For each of the cases \ref{I}-\ref{VIII}, we provide the corresponding bifurcation diagram.
We notice, however, that our diagrams are only illustrative and ``minimal''. That is, one should not deduce directly from them the precise number of solutions for a given value of the parameter $\lambda$. 
In some cases, namely, \ref{I} and \ref{VI}, we obtain only nonnegative solutions and it would be interesting to know if they are positive or of compact support type. 
This matter will be pursued elsewhere.

\section{Preliminaries and auxiliary results}\label{section:prelim}
Along this section, we always assume that $m,n>0$, and, in order to keep generality, we \textit{do not} impose the restriction $m \leq 1$ from \eqref{eq:mn} unless otherwise stated explicitly.
Hereinafter, we use the standard notation $L^p(\Omega)$ for the Lebesgue spaces, $1 \leq p \leq \infty$, and denote by $\|\cdot\|_{p}$ the associated norm. 
By $W_0^{1,2}(\Omega)$ we denote the standard Sobolev space endowed with the norm $\|\nabla u\|_{2} := \left(\int_\Omega |\nabla u|^2 \, dx\right)^\frac{1}{2}$, and  $2^*$ stands for the critical Sobolev exponent, i.e., $2^*=\frac{2N}{N-2}$ provided $N>2$, and $2^*=+\infty$ provided $N=1$. With a slight abuse of notation, we set $2^*=+\infty$ also in the case $N=2$. 
Let us denote by $X$ the space $W_0^{1,2}(\Omega) \cap L^{\max\{m+2, n+1\}}(\Omega)$ endowed with the norm $\|u\|_X := \|\nabla u\|_2 + \|u\|_{\max\{m+2, n+1\}}$. 
Notice that $X$ is a reflexive Banach space. Moreover, $X \equiv W_0^{1,2}(\Omega)$ provided $\max\{m+2, n+1\} < 2^*$. 
We will denote by $\lambda_1$ and $\varphi_1$ the first eigenvalue and the corresponding first eigenfunction of the Dirichlet Laplace operator on $\Omega$, respectively. 
We assume, without loss of generality, that $\varphi_1>0$ in $\Omega$ and $\|\varphi_1\|_\infty = 1$.

The problem \eqref{eq:D} has a variational structure and the associated energy functional $E_\lambda \in C^1(X, \mathbb{R})$ is given by
$$
E_\lambda(u) = 
\frac{1}{2} \int_\Omega |\nabla u|^2 \, dx 
-
\frac{1}{m+1} \int_\Omega |u^+|^{m+1} \, dx
+
\frac{1}{m+2} \int_\Omega |u^+|^{m+2} \, dx  
+
\frac{\lambda}{n+1} \int_\Omega |u^+|^{n+1} \, dx,
$$
where $u^+ = \max\{u, 0\} \in X$.
Evidently, the weak lower-semicontinuity of the norms of $X$ and $L^p(\Omega)$ and the Rellich-Kondrachov theorem imply that $E_\lambda$ is weakly lower-semicontinuous on $X$ if $m < 2^*-1$ and, additionally, any of the following assumptions is satisfied:
\begin{alignat*}{2}
&\lambda < 0, \quad&&n < 2^*-1;\\
&\lambda \geq -\frac{n+1}{m+2}, \quad&&n = m+1;\\
&\lambda \geq 0.  \quad&&
\end{alignat*}
Each critical point $u \in X$ of $E_\lambda$ satisfies $u \geq 0$ in $\Omega$ since
$$
0 = \left< E_\lambda'(u), u^- \right> = \int_{\Omega} |\nabla u^-|^2 \,dx,
$$
where $u^- = \max\{-u, 0\} \in X$, which yields $u^- \equiv 0$ in $\Omega$. 
By definition, nonzero critical points of $E_\lambda$ are weak solutions to \eqref{eq:D}. 
Notice that apart from nonnegative solutions the equation in \eqref{eq:D} with zero Dirichlet boundary conditions may also possess sign-changing solutions, although we are not interested in latter ones in the present work.

\begin{lemma}\label{lem:reg}
	Let $m<2^*-1$.
	Assume that either $\lambda < 0$ and $n < 2^*-1$ or $\lambda \geq 0$. 
	Then any weak solution to \eqref{eq:D} is a classical solution. 
\end{lemma}
\begin{proof}
	Let $u \in X$ be a weak solution to \eqref{eq:D}. 
	By the assumptions, the term $\frac{1}{m+1} \int_\Omega |u^+|^{m+1} \, dx$ presented in $E_\lambda(u)$ is subcritical, and the same is true for $\frac{\lambda}{n+1} \int_\Omega |u^+|^{n+1} \, dx$ if $\lambda < 0$.
	On the other hand, the term $\frac{1}{m+2} \int_\Omega |u^+|^{m+2} \, dx$ is with a positive sign, and the same applies to $\frac{\lambda}{n+1} \int_\Omega |u^+|^{n+1} \, dx$ if $\lambda \geq 0$. 
	Thus, carrying out the standard bootstrap argument (see, e.g., \cite[Lemma 3.2]{DKN}), we appeal to the Sobolev embedding theorem with regard to two former terms and estimate from below two latter terms by zero, and hence derive that $u \in L^\infty(\Omega)$. 
	Therefore, by the $L^p$-regularity results in \cite{ADN} or \cite{giltrud}, $u \in W^{2,p}(\Omega)$ for any $1 < p < \infty$ and then in $C^{1,\alpha}(\overline{\Omega})$ for any $\alpha \in (0,1)$.
	Finally, applying again the regularity theory from \cite{giltrud}, we deduce that $u \in C^{2,\beta}(\Omega)$ for some $\beta \in (0,1)$.
\end{proof}

For any $v \in X \setminus \{0\}$, we define fibers $\phi_v(t)$, $t>0$, associated with $E_\lambda(v)$ as 
$$
\phi_v(t) = \frac{t^2}{2} \int_\Omega |\nabla v|^2 \, dx
-
\frac{t^{m+1}}{m+1} \int_\Omega |v^+|^{m+1} \, dx 
+
\frac{t^{m+2}}{m+2} \int_\Omega |v^+|^{m+2} \, dx
+
\frac{\lambda t^{n+1}}{n+1} \int_\Omega |v^+|^{n+1} \, dx.
$$

We say that $\overline{u} \in C^2(\Omega) \cap C(\overline{\Omega})$ is a supersolution to \eqref{eq:D} whenever $\overline{u} \geq 0$ in $\Omega$ and 
$$
\left\{
\begin{aligned}
-\Delta \overline{u} &\geq (1 - \overline{u}) \overline{u}^m - \lambda \overline{u}^n &&\text{in}~ \Omega,\\
\overline{u} &\geq 0 &&\text{on}~ \partial\Omega.
\end{aligned}
\right.
$$
Similarly, we say that $\underline{u} \in C^2(\Omega) \cap C(\overline{\Omega})$ is a subsolution to \eqref{eq:D} whenever $\underline{u} \geq 0$ in $\Omega$ and 
$$
\left\{
\begin{aligned}
-\Delta \underline{u} &\leq (1 - \underline{u}) \underline{u}^m - \lambda \underline{u}^n &&\text{in}~ \Omega,\\
\underline{u} &= 0 &&\text{on}~ \partial\Omega.
\end{aligned}
\right.
$$

Let us provide several auxiliary results which will be useful in what follows. 
The first lemma gives a sufficient condition for the nonexistence of compact support solutions to \eqref{eq:D}, and the second one provides a uniform $L^\infty$-bound for solutions to \eqref{eq:D}.
\begin{lemma}\label{lem:smp}
	Assume that one of the following assumptions is satisfied:
	\begin{alignat}{2}
	\label{eq:lem:smp1}
	&\lambda \leq 0;  \quad&&\\
	\label{eq:lem:smp2}
	&\lambda > 0, \quad&&n \geq 1;\\
	\label{eq:lem:smp3}
	&\lambda > 0, \quad&&n > m;\\
	\label{eq:lem:smp4}
	0<\, &\lambda < 1, \quad&&n=m.
	\end{alignat}
	Let $u$ be a solution to \eqref{eq:D}.
	Then $u>0$ in $\Omega$ and $\frac{\partial u}{\partial \nu} < 0$ on $\partial \Omega$.
\end{lemma}
\begin{proof}
	Assuming that $u$ is a solution to \eqref{eq:D}, we denote by $v$ a copy of $u$, that is, $v \equiv u$ in $\Omega$.
	
	Let $\lambda \leq 0$. The problem \eqref{eq:D} can be written as 
	\begin{equation*}\label{eq:lemD2}
	\left\{
	\begin{aligned}
	-\Delta u + v^{m} u &= v^m - \lambda v^n \geq 0 &&\text{in}~ \Omega,\\
	u &= 0 &&\text{on}~ \partial\Omega,
	\end{aligned}
	\right.
	\end{equation*}
	and we apply the strong maximum principle (see, e.g., \cite[Section 3.2]{giltrud}) to this linear equation to obtain the desired result.
	
	If $\lambda > 0$ and $n \geq 1$, then $u$ satisfies the problem 
	\begin{equation*}\label{eq:lemD}
	\left\{
	\begin{aligned}
	-\Delta u + (\lambda v^{n-1} + v^{m}) u &= v^m \geq 0	&&\text{in}~ \Omega,\\
	u &= 0 &&\text{on}~ \partial\Omega,
	\end{aligned}
	\right.
	\end{equation*}
	where $\lambda v^{n-1} + v^{m} \geq 0$ in $\Omega$, and again the strong maximum principle gives the result.
	
	Assume that $\lambda>0$ and $n>m$. 
	Then $u$ satisfies
	\begin{equation*}
	\left\{
	\begin{aligned}
	-\Delta u &= v^m (1-v-\lambda v^{n-m})	&&\text{in}~ \Omega,\\
	u &= 0 &&\text{on}~ \partial\Omega.
	\end{aligned}
	\right.
	\end{equation*}
	Notice that for any $\lambda>0$ there exists $r>0$ such that $1-r-\lambda r^{n-m} > 0$.
	Suppose, by contradiction, that $u(x_0)=0$ for some $x_0 \in \Omega$. 
	Defining $A = \{x \in \Omega: u(x) < r\}$, we see that $A$ is a nonempty open set containing $x_0$, $-\Delta u \geq 0$ in $A$, $u=r$ on $\partial A \cap \Omega$, and $u=0$ on $\partial A \cap \partial \Omega$.
	Therefore, the strong maximum principle yields $u>0$ in $A$, a contradiction.
	
	Finally, if $0<\lambda<1$ and $n=m$, we write the problem \eqref{eq:D} as
	\begin{equation*}
	\left\{
	\begin{aligned}
	-\Delta u + v^{m} u &= (1-\lambda) v^m  \geq 0 &&\text{in}~ \Omega,\\
	u &= 0 &&\text{on}~ \partial\Omega,
	\end{aligned}
	\right.
	\end{equation*}
	and the strong maximum principle gives $u>0$ in $\Omega$ and $\frac{\partial u}{\partial n} < 0$ on $\partial \Omega$. 
\end{proof}

\begin{remark}\label{rem:smp}
	We see from Lemma \ref{lem:smp} that $n < \min\{m, 1\}$ is a \textit{necessary} condition in order to obtain a compact support solution to \eqref{eq:D} with $\lambda>0$.	
\end{remark}

\begin{lemma}\label{lem:bound}
	Assume that one of the following assumptions is satisfied:
	\begin{alignat}{2}
	\label{eq:bound3}
	-1 <\, &\lambda < 0, \quad &&n = m + 1;\\
	\label{eq:bound2}
	&\lambda < 0, \quad &&n < m+ 1;\\
	\label{eq:bound1}
	&\lambda \geq 0. \quad &&
	\end{alignat}
	If $u$ is a solution to \eqref{eq:D}, then there exists $M(\lambda)>0$ such that
	$$
	\|u\|_\infty \leq M(\lambda).
	$$
	In particular, in the case \eqref{eq:bound1}, $M(\lambda)=1$. If, in addition to \eqref{eq:bound1}, $\lambda>0$ and $n>m$, then $M(\lambda) = \min\{1, \lambda^\frac{1}{m-n}\}$.
\end{lemma}
\begin{proof}
	Let $M := \max\limits_{x \in \Omega} u(x) = u(x_0) > 0$ for some $x_0 \in \Omega$, that is, $\|u\|_\infty = M$.
	Then
	$$
	0 \leq -\Delta u(x_0) = M^m - M^{m+1}  - \lambda M^n,
	$$
	and hence the desired bounds follow easily.
\end{proof}

We next state general existence results for \eqref{eq:D} obtained by application of the sub and supersolution method.
\begin{thm}\label{thm:existence}
	Assume that one of the following assumptions is satisfied: 
	\begin{alignat}{3}
	\label{eq:them:exist0}
	-1 <\, &\lambda < 0,  \quad &&m < 1, \quad &&n=m+1;\\
	\label{eq:them:exist2}
	&\lambda < 0,  \quad &&m < 1, \quad &&n<m+1;\\
	\label{eq:them:exist1}
	&\lambda < 0,  \quad && \quad &&n<1;\\
	\label{eq:them:exist3}
	&\lambda = 0,  \quad &&m < 1; \quad &&\\
	\label{eq:them:exist4}
	&\lambda > 0,  \quad &&m < 1, \quad &&n>m;\\
	\label{eq:them:exist5}
	0 < \, &\lambda < 1,  \quad &&m < 1, \quad &&n=m.
	\end{alignat}
	Then \eqref{eq:D} possesses a positive solution. 
\end{thm}
\begin{proof}
	Considering the function $\overline{u} \equiv M$ in $\Omega$, where $M>0$ is a constant, we easily see that 
	$$
	-\Delta \overline{u} = 0 \geq M^m - M^{m+1} - \lambda M^n \equiv (1 - \overline{u}) \overline{u}^m - \lambda \overline{u}^n 
	\quad \text{in}~ \Omega
	$$ 
	for any sufficiently large $M$ under either of the assumptions \eqref{eq:them:exist0}-\eqref{eq:them:exist5}. That is, $\overline{u}$ is a supersolution to \eqref{eq:D}. 
	Considering the function $\underline{u} = c \varphi_1$, we obtain
	$$
	-\Delta \underline{u} 
	= 
	\lambda_1 c \varphi_1 
	\leq 
	c^m \varphi_1^m - c^{m+1} \varphi_1^{m+1} - \lambda c^n \varphi_1^n 
	\equiv
	(1- \underline{u})\underline{u}^m - \lambda \underline{u}^n 
	\quad \text{in}~ \Omega
	$$
	for any sufficiently small $c>0$, provided any of the assumptions \eqref{eq:them:exist0}-\eqref{eq:them:exist5} holds. Thus, $\underline{u}$ is a subsolution to \eqref{eq:D}. 
	Therefore, applying the sub and supersolution method with $c\varphi_1 \leq M$ in $\Omega$ and recalling that $\varphi_1>0$ in $\Omega$, we obtain the existence of a positive solution to \eqref{eq:D}. 
\end{proof}
\begin{thm}\label{thm:existence2}
	Let $m=1$.
	Assume that one of the following assumptions is satisfied: 
	\begin{alignat}{3}
	\label{eq:them:exist0m}
	&\lambda = 0,  \quad &&\lambda_1 < 1; \quad &&\\
	\label{eq:them:exist2m}
	&\lambda < 0,  \quad &&\lambda_1 \leq 1, \quad &&n<2;\\
	\label{eq:them:exist1m}
	&\lambda > 0,  \quad &&\lambda_1 < 1, \quad &&n>1.
	\end{alignat}
	Then \eqref{eq:D} possesses a positive solution. 
	Moreover, if $\lambda_1 \geq 1$ and \eqref{eq:D} possesses a solution, then $\lambda<0$.
\end{thm}
\begin{proof}
	As in Theorem \ref{thm:existence}, we consider the function $\overline{u} \equiv M$ in $\Omega$, where $M>0$ is a constant, and readily see that
	$$
	-\Delta \overline{u} = 0 \geq M - M^{2} - \lambda M^n \equiv (1 - \overline{u}) \overline{u} - \lambda \overline{u}^n 
	\quad \text{in}~ \Omega
	$$ 
	for any sufficiently large $M$ under either of the assumptions \eqref{eq:them:exist0m}-\eqref{eq:them:exist1m}. 
	Thus, $\overline{u}$ is a supersolution to \eqref{eq:D}. 
	On the other hand, taking the function $\underline{u} = c \varphi_1$, we again get
	$$
	-\Delta \underline{u} 
	= 
	\lambda_1 c \varphi_1 
	\leq 
	c \varphi_1 - c^{2} \varphi_1^{2} - \lambda c^n \varphi_1^n 
	\equiv
	(1 - \underline{u})\underline{u} - \lambda \underline{u}^n 
	\quad \text{in}~ \Omega
	$$
	for any sufficiently small $c>0$, provided any of the assumptions \eqref{eq:them:exist0m}-\eqref{eq:them:exist1m} holds. 
	That is, $\underline{u}$ is a subsolution to \eqref{eq:D}. 
	Applying the sub and supersolution method with $c \varphi_1 \leq M$ in $\Omega$, we derive the existence of a positive solution to \eqref{eq:D}. 
	
	Assume now that $\lambda_1 \geq 1$, and \eqref{eq:D} possesses a solution $u$ for some $\lambda \geq 0$. 
	Consider the fibers
	\begin{equation}\label{eq:fibers:2}
	\phi_{u}(t) 
	= \frac{t^2}{2} \left(\int_\Omega |\nabla u|^2 \, dx - \int_\Omega u^{2} \, dx \right)
	+
	\frac{t^{3}}{3} \int_\Omega u^{3} \, dx  
	+
	\frac{\lambda t^{n+1}}{n+1} \int_\Omega u^{n+1} \, dx
	\end{equation}
	associated with $E_\lambda(u)$. Since $u$ is a solution, we must have $\left.\frac{d \phi_u(t)}{dt}\right|_{t=1} = 0$, which yields 
	$$
	\int_\Omega |\nabla u|^2 \, dx - \int_\Omega u^{2} \, dx  < 0.
	$$
	Hence, by the variational characterization of $\lambda_1$, we conclude that $\lambda_1 < 1$, a contradiction.
\end{proof}

We now state a uniqueness result for \eqref{eq:D}.
\begin{thm}\label{thm:uniq}
	Let $m\leq 1$. 
	Assume that one of the following assumptions is satisfied: 
	\begin{alignat}{2}
	\label{eq:them:uniq0}
	&\lambda > -1,  \quad &&n = m+ 1;\\
	\label{eq:them:uniq2}
	&\lambda \geq 0,  \quad &&n \geq 1;\\
	\label{eq:them:uniq3}
	&\lambda > 0,  \quad &&n > m;\\
	\label{eq:them:uniq4}
	0<\,&\lambda < 1, \quad &&n = m;\\
	\label{eq:them:uniq1}
	&\lambda \leq 0,  \quad &&n \leq 1.
	\end{alignat}
	Then \eqref{eq:D} possesses at most one positive solution.
\end{thm}
\begin{proof}	
	Assume that there exists a solution $u$ to \eqref{eq:D} such that $u>0$ in $\Omega$. 
	The uniqueness of this solution follows by applying the method of \cite[Section 2]{brezisoswald}, see also \cite{HMV}. Indeed, to use the result of \cite{brezisoswald} it is enough to show that $\frac{f(s)}{s}$ is (strictly) decreasing for $s>0$, where $f(s)$ is the nonlinearity in \eqref{eq:D}, i.e.,
	$$
	f(s) = s^m - s^{m+1} - \lambda s^n.
	$$
	Since $f(s)$ is smooth for $s>0$, we see that the monotonicity condition for $\frac{f(s)}{s}$ is equivalent to showing that
	\begin{equation}\label{eq:fss>0}
	f(s) - s f'(s) = (1-m) s^m + m s^{m+1} + \lambda (n-1) s^n > 0
	\quad 
	\text{for all}~
	s>0.
	\end{equation}
	Under either of the assumptions \eqref{eq:them:uniq0}, \eqref{eq:them:uniq2}, \eqref{eq:them:uniq4}, \eqref{eq:them:uniq1}, the validity of \eqref{eq:fss>0} can be justified directly.
	Let us show that \eqref{eq:fss>0} holds true also under the assumption \eqref{eq:them:uniq3}. 
	Recall that, in view of Lemma \ref{lem:bound}, under this assumption any solution $u$ to \eqref{eq:D} satisfies $\|u\|_\infty \leq \lambda^\frac{1}{m-n}$. Thus, it is enough to obtain \eqref{eq:fss>0} only for $s \in \left(0, \lambda^\frac{1}{m-n}\right]$.
	Dividing by $s^n$, we see that \eqref{eq:fss>0} is equivalent to
	$$
	(1-m)s^{m-n} + m s^{m-n+1} > -\lambda (n-1).
	$$
	Using the bound $s^{m-n} \geq \lambda$, we conclude that \eqref{eq:fss>0} follows from the inequality
	$$
	m s^{m-n+1} > -\lambda (n-m)
	$$
	which is trivially satisfied for $s>0$. The proof is complete.
\end{proof}

Finally, we show that $E_\lambda$ satisfies the Palais-Smale condition under several assumptions on $m$ and $n$ which will be used below.
\begin{lemma}\label{lem:PS}
	Let $m<2^*-2$ and $n<2^*-1$.
	Assume that one of the following assumptions is satisfied:
	\begin{alignat}{3}
	\label{eq:lem:PS1}
	&\lambda \geq 0,  \quad &&m < 1, \quad &&n \leq m+1;\\
	\label{eq:lem:PS2}
	&						&&m < n-1. &&
	\end{alignat}
	Then $E_\lambda$ satisfies the Palais-Smale condition on $W_0^{1,2}(\Omega)$. Namely, if $\{u_n\} \subset W_0^{1,2}(\Omega)$ and there exists $C>0$ such that $|E_\lambda(u_n)| \leq C$ for all $n \in \mathbb{N}$, and $\|E_\lambda'(u_n)\| \to 0$ as $n \to +\infty$, then $\{u_n\}$ has a strongly convergent subsequence.
\end{lemma}
\begin{proof}
	To prove our claim under the assumption \eqref{eq:lem:PS1}, we use the Sobolev embedding theorem to get
	\begin{align*}
	&C + o(1) \left(\int_\Omega |\nabla u_n|^2 \,dx\right)^\frac{1}{2} 
	\geq (m+2) E_\lambda(u_n) - \left< E_\lambda'(u_n), u_n \right> 
	\\
	&=
	\frac{m}{2} \int_\Omega |\nabla u_n|^2 \,dx 
	- 
	\frac{1}{m+1} \int_\Omega |u_n^+|^{m+1} \, dx
	+ 
	\frac{\lambda (m-n+1)}{n+1} \int_\Omega |u_n^+|^{n+1} \, dx
	\\
	&\geq 
	\frac{m}{2} \int_\Omega |\nabla u_n|^2 \,dx 
	-
	C \left(\int_\Omega |\nabla u_n|^2 \,dx\right)^\frac{m+1}{2}, 
	\end{align*}
	which implies that $\|\nabla u_n\|_2 < C$ for all $n \in \mathbb{N}$, and hence $u_n$ converges weakly in $W_0^{1,2}(\Omega)$ and strongly in $L^r(\Omega)$, $r \in [1,2^*)$, to some $u_0$, up to a subsequence. 
	Therefore, applying \cite[Chapter II, Proposition 2.2]{struwe}, we conclude that $u_n \to u_0$ strongly in $W_0^{1,2}(\Omega)$.
	
	To prove our claim under the assumption \eqref{eq:lem:PS2}, we use the H\"older inequality to get
	\begin{align*}
	&C + o(1) \left(\int_\Omega |\nabla u_n|^2 \,dx\right)^\frac{1}{2} 
	\geq (n+1) E_\lambda(u_n) - \left< E_\lambda'(u_n), u_n \right> 
	\\
	&=
	\frac{n-1}{2} \int_\Omega |\nabla u_n|^2 \,dx 
	-
	\frac{n-m}{m+1} \int_\Omega |u_n^+|^{m+1} \, dx
	+ 
	\frac{n-m-1}{m+2} \int_\Omega |u_n^+|^{m+2} \, dx
	\\
	&\geq 
	\frac{n-1}{2} \int_\Omega |\nabla u_n|^2 \,dx 
	-
	\frac{n-m}{m+1} \int_\Omega |u_n^+|^{m+1} \, dx
	+
	C \left( \int_\Omega |u_n^+|^{m+1} \, dx\right)^\frac{m+2}{m+1},
	\end{align*}
	which implies that $\|\nabla u_n\|_2 < C$ for all $n \in \mathbb{N}$, and the rest of the proof follows as above.
\end{proof}

\section{Case \ref{I}}\label{sec:I}
In this section, we consider the case $0<n< m<1$.
\begin{thm}\label{thm:caseI}
	Let $0<n<m<1$. Then there exists $\lambda^* > 0$ satisfying
	\begin{equation}\label{eq:caseIbound}
	\lambda^* \leq \frac{(m-n)^{m-n}}{(m-n+1)^{m-n+1}} < 1,
	\end{equation}
	such that the problem \eqref{eq:D} has a solution $u_\lambda$ for any $\lambda < \lambda^*$ and has no solution for any $\lambda > \lambda^*$. 
	If $\lambda \leq 0$, then $u_\lambda$ is positive in $\Omega$ and unique. 
	Moreover, there exists $\lambda_* \in (0, \lambda^*)$ such that $E_\lambda(u_\lambda) < 0$ for any $\lambda \in (0, \lambda_*)$.
\end{thm}
\begin{proof}
	The existence, positivity, and uniqueness in the case $\lambda \leq 0$ are given by Theorems \ref{thm:existence} and \ref{thm:uniq}. 
	Let us thus assume that $\lambda > 0$. 
	We will obtain a solution to \eqref{eq:D} by global minimization of the energy functional $E_\lambda$. Recall that $E_\lambda$ is weakly lower-semicontinuous on $X$. Moreover, $E_\lambda$ is coercive on $X$. Indeed, applying the Sobolev embedding theorem to the term $\int_{\Omega} |u^+|^{m+1} \, dx$, one can find $C>0$ (independent of $u$) such that for all $u \in X$, 
	\begin{align*}
	E_\lambda(u) 
	> \frac{1}{2} \int_\Omega |\nabla u|^2 \, dx 
	- C \left(\int_\Omega |\nabla u^+|^2 \, dx\right)^\frac{m+1}{2}
	+
	\frac{1}{m+2} \int_\Omega |u^+|^{m+2} \, dx. 
	\end{align*}
	Thus, since $0<m<1$, we easily see that $E_\lambda(u) \to +\infty$ provided $\|u\|_X \to +\infty$, which is the desired coercivity.
	Let us show now that for any sufficiently small $\lambda > 0$ there exists $u$ such that $E_\lambda(u) < 0$. 
	Let us take any $v \in X \setminus \{0\}$ such that $v \geq 0$ in $\Omega$ and consider the fibers associated with $E_0(v)$:
	$$
	\phi_v(t) = \frac{t^2}{2} \int_\Omega |\nabla v|^2 \, dx
	-
	\frac{t^{m+1}}{m+1} \int_\Omega |v|^{m+1} \, dx 
	+
	\frac{t^{m+2}}{m+2} \int_\Omega |v|^{m+2} \, dx,
	\quad t>0.
	$$
	Recalling that $0<m<1$, we deduce that the term $-\frac{t^{m+1}}{m+1} \int_\Omega |v|^{m+1} \, dx$ is leading as $t \to 0$, and hence $\phi_v(t) < 0$ for all sufficiently small $t > 0$. Let us fix any of such $t>0$. 
	Due to the continuity of $E_\lambda(tv)$ with respect to $\lambda$, we conclude that $E_\lambda(tv)<0$ for all sufficiently small $\lambda>0$.
	Finally, applying the direct minimization procedure (see, e.g., \cite[Chapter I, Theorem 1.2]{struwe}), we obtain the existence of a critical point $u_\lambda \geq 0$ of $E_\lambda$ such that $E_\lambda(u_\lambda)<0$ for all sufficiently small $\lambda>0$. 
	Denoting
	\begin{align*}
	\lambda^* &= \sup\{\lambda>0:~ \eqref{eq:D} ~\text{has a solution}\},\\
	\lambda_* &= \sup\{\lambda>0:~ \eqref{eq:D} ~\text{has a solution $u_\lambda$ such that $E_\lambda(u_\lambda)<0$}\},
	\end{align*}
	we get $\lambda^* \geq \lambda_* > 0$, and the proof of the existence part of the theorem is finished.
	Notice that the same arguments as above provide the existence also in the case $\lambda \leq 0$.	
	
	Now we prove the bound \eqref{eq:caseIbound} for $\lambda^*$. 
	Let $u_\lambda$ be a solution to \eqref{eq:D} and denote $M_\lambda = \max\limits_{x \in \Omega} u_\lambda(x) = u_\lambda(x_\lambda) > 0$ for some $x_\lambda \in \Omega$, that is, $\|u_\lambda\|_\infty = M_\lambda$.
	Then
	\begin{equation}\label{eq:caseI1}
	0 \leq -\Delta u_\lambda(x_\lambda) = M_\lambda^n \left(M_\lambda^{m-n} - M_\lambda^{m+1-n}  - \lambda\right).
	\end{equation}
	Let us investigate the function $g_\lambda(s) = s^{m-n} - s^{m+1-n}  - \lambda$ for $s>0$. 
	Since $0<n<m<1$, we see that $s^{m-n}-\lambda$ is the leading term as $s \to 0$ and $-s^{m+1-n}$ is the leading term as $s \to +\infty$. 
	Moreover, $g''_\lambda(s) < 0$ for all $s>0$.
	Therefore, $g_\lambda(s)$ has exactly one critical point for $s>0$ which is the point of global maximum. Let us denote this point as $s_\lambda$. 
	Performing direct calculations, one can derive that the map $\lambda \mapsto g_\lambda(s_\lambda)$, $\lambda \in (0,+\infty)$, is decreasing. 
	Looking for $\overline{\lambda}$ such that $g_{\overline{\lambda}}(s_{\overline{\lambda}}) = 0$, we get
	$$
	\overline{\lambda} = \frac{(m-n)^{m-n}}{(m-n+1)^{m-n+1}} < 1.
	$$
	Since, $M_\lambda^n g_\lambda(M_\lambda) \geq 0$ by \eqref{eq:caseI1}, we must have $s_\lambda^n g_\lambda(s_\lambda) \geq 0$, which implies $\lambda \leq \overline{\lambda}$, and hence $\lambda^* \leq \overline{\lambda}$. 
\end{proof}

\begin{figure}[ht]
	\centering
	\includegraphics[width=0.6\linewidth]{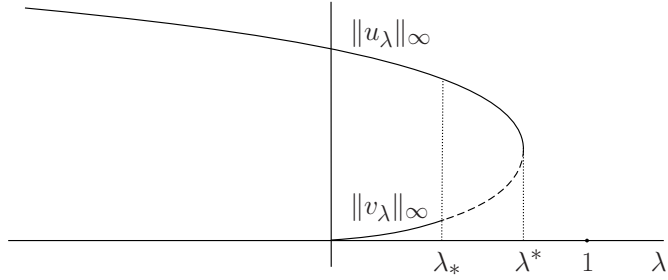}\\
	\caption{The branches of solutions to \eqref{eq:D} in the case $0<n<m<1$.}
	\label{fig:caseI}
\end{figure}

Let us now discuss a multiplicity issue for $\lambda>0$.
Considering the fibers $\phi_v(t)$ associated with $E_\lambda(v)$ for any $v \in X \setminus \{0\}$, we see that $\phi_v(t)$ possess exactly two critical points for all sufficiently small $\lambda>0$, namely, a point of maximum and a point of minimum.
In fact, the solution $u_\lambda$ obtained in Theorem \ref{thm:caseI} corresponds to the point of minimum of $\phi_{u_\lambda}(t)$. 
In view of the presence of the point of maximum of $\phi_v(t)$ it is natural to anticipate the existence of another critical point of $E_\lambda$ which has a mountain pass property. 
Indeed, in the following theorem we show the existence of the second critical point of $E_\lambda$ by obtaining a suitable neighbourhood of $u_\lambda$ with a lower bound for the mountain pass range on its boundary.

\begin{thm}\label{thm:caseIb}
	Let $0<n<m<1$, $m<2^*-2$, and 
	\begin{equation}\label{eq:caseIas}
	n < \frac{(N+2)(2m-1)-(N-2)m^2}{4}.
	\end{equation}
	Let $\lambda_*$ be defined as in Theorem \ref{thm:caseI}.
	Then, for any $\lambda \in (0, \lambda_*)$, the problem \eqref{eq:D} possesses a solution $v_\lambda$ such that $E_\lambda(v_\lambda)>0$.
\end{thm}
\begin{proof}
	Along the proof we will denote by $C>0$ various constants which do not depend on $u \in W_0^{1,2}(\Omega)$. 
	Let us fix $\lambda \in (0, \lambda_*)$. By Theorem \ref{thm:caseI}, there exists a critical point $u_\lambda$ of $E_\lambda$ such that $E_\lambda(u_\lambda)<0$.
	Let us define
	$$
	c_\lambda = \inf_{g \in \Gamma} \max_{0 \leq t \leq 1} E_\lambda(g(t)),
	$$
	where $\Gamma$ is the set of all continuous paths $g:[0,1] \to W_0^{1,2}(\Omega)$ satisfying $g(0)=u_\lambda$ and $g(1)=0$. 
	Clearly, $c_\lambda \geq 0$ since $E_\lambda(g(1)) = 0$.
	Let us show that $c_\lambda > 0$.
	To this end, consider the set
	$$
	F_\rho 
	= 
	\left\{u \in W_0^{1,2}(\Omega):~ 
	\frac{1}{m+1} \int_\Omega |u^+|^{m+1} \, dx 
	-
	\frac{1}{m+2} \int_\Omega |u^+|^{m+2} \, dx 
	> 
	\frac{1}{2} \int_\Omega |\nabla u|^2 \, dx 	> \frac{\rho^2}{2}\right\}
	$$
	for some sufficiently small $\rho \in (0,\|\nabla u_\lambda\|_2)$. 
	This set is bounded in $W_0^{1,2}(\Omega)$ in view of the Sobolev embedding theorem since $m<2^*-2$. 	
	Recalling that $E_\lambda(u_\lambda)<0$ and $\lambda>0$, we get $u_\lambda \in F_\rho$. 
	Moreover, the closure of $F_\rho$ does not contain $0$. 
	Let us prove that there exists $\alpha>0$ such that $E_\lambda(u)>\alpha$ for any $u \in \partial F_\rho$. 
	Fixing any $u \in \partial F_\rho$, we have either 
	\begin{equation}\label{eq:Fboundary1}
	\frac{1}{m+1}\int_\Omega |u^+|^{m+1} \, dx 
	-
	\frac{1}{m+2} \int_\Omega |u^+|^{m+2} \, dx 
	= 
	\frac{1}{2} \int_\Omega |\nabla u|^2 \, dx \geq \frac{\rho^2}{2}
	\end{equation}
	or
	\begin{equation}\label{eq:Fboundary2}
	\frac{1}{m+1}\int_\Omega |u^+|^{m+1} \, dx 
	-
	\frac{1}{m+2} \int_\Omega |u^+|^{m+2} \, dx 
	\geq 
	\frac{1}{2} \int_\Omega |\nabla u|^2 \, dx = \frac{\rho^2}{2}.
	\end{equation}
	Evidently, in both cases,
	\begin{equation}\label{eq:um+1}
	\int_\Omega |u^+|^{m+1} \, dx > \frac{(m+1)\rho^2}{2}.
	\end{equation}
	Let us also estimate $\int_\Omega |u^+|^{n+1} \, dx$ from below.
	The Gagliardo-Nirenberg inequality (see, e.g, \cite{nirenberg}) gives
	\begin{align}
	\notag
	\left(\int_\Omega |u^+|^{m+1} \, dx\right)^\frac{1}{m+1} 
	&\leq
	C
	\left(\int_\Omega |u^+|^{n+1} \, dx\right)^\frac{\theta}{n+1}
	\left(\int_\Omega |\nabla u^+|^{2} \, dx\right)^\frac{1-\theta}{2}
	\\
	\label{eq:GN}
	&\leq
	C
	\left(\int_\Omega |u^+|^{n+1} \, dx\right)^\frac{\theta}{n+1} \rho^{1-\theta}
	,
	\end{align}
	where $\theta \in (0,1)$ satisfies
	$$
	\frac{1}{m+1} = (1-\theta)\left(\frac{1}{2}-\frac{1}{N}\right) + \frac{\theta}{n+1}.
	$$
	Therefore, using \eqref{eq:um+1}, we get from \eqref{eq:GN} that
	\begin{equation}\label{eq:un+1}
	\int_\Omega |u^+|^{n+1} \, dx \geq C \rho^{\frac{n+1}{\theta}\left(\frac{2}{m+1}-1+\theta\right)}.
	\end{equation}	
	Assume first that \eqref{eq:Fboundary1} is satisfied. 
	Then \eqref{eq:un+1} yields
	\begin{equation}\label{eq:I>01}
	E_\lambda(u) 
	=
	\frac{\lambda}{n+1} \int_\Omega |u^+|^{n+1} \, dx
	\geq 	
	C \rho^{\frac{n+1}{\theta}\left(\frac{2}{m+1}-1+\theta\right)}.
	\end{equation}
	Assume now that \eqref{eq:Fboundary2} is satisfied. 
	Applying the Sobolev embedding theorem, we get
	\begin{equation}\label{eq:um+12}
	\int_\Omega |u^+|^{m+1} \, dx \leq C \left(\int_\Omega |\nabla u^+|^2\right)^\frac{m+1}{2} \leq C \rho^{m+1}.
	\end{equation}
	Thus, using \eqref{eq:un+1} and \eqref{eq:um+12}, we estimate $E_\lambda(u)$ from below as follows:
	\begin{equation}\label{eq:Elowerbound1}
	E_\lambda(u) \geq \frac{\rho^2}{2} - C\rho^{m+1} + C \rho^{\frac{n+1}{\theta}\left(\frac{2}{m+1}-1+\theta\right)}.
	\end{equation}
	Notice that the assumption \eqref{eq:caseIas} is, in fact, equivalent to the inequality
	$$
	m+1 > \frac{n+1}{\theta}\left(\frac{2}{m+1}-1+\theta\right).
	$$
	Therefore, if follows from \eqref{eq:Elowerbound1} that there is $C>0$, which depends on $\rho$ but does not depend on $u$ satisfying \eqref{eq:Fboundary2}, such that $E_\lambda(u) > C>0$ for any sufficiently small $\rho>0$. 
	Combining this fact with the estimate \eqref{eq:I>01}, we conclude that for any sufficiently small $\rho>0$ there exists $\alpha>0$ such that $E_\lambda(u)>\alpha$ for any $u \in \partial F_\rho$. 
	Since any continuous path in $W_0^{1,2}(\Omega)$ joining $u_\lambda$ and $0$  crosses $\partial F_\rho$, we conclude that $c_\lambda > 0$.
	 
	At the same time, $E_\lambda$ satisfies the Palais-Smale condition on $W_0^{1,2}(\Omega)$ by Lemma \ref{lem:PS}. 
	Therefore, the mountain pass theorem (see, e.g., \cite{rabinowitz}) provides us with the critical point $v_\lambda$ of $E_\lambda$ such that $c_\lambda = E_\lambda(v_\lambda)>0$.
\end{proof}

\begin{remark}
	One can show that the set of exponents $n, m$ admissible in Theorem \ref{thm:caseIb} is nonempty for $N \in [1,8]$ and empty for $N \geq 9$. 
	Moreover, the function on the right-hand side of \eqref{eq:caseIas} is (strictly) increasing with respect to $m \in (0, 1)$.
\end{remark}

\begin{remark}\label{rem:caseI:compact}
	The results of \textsc{Franchi, Lanconelli, \& Serrin} \cite{FLS} or \textsc{Gazzola, Serrin, \& Tang} \cite{GST} can be applied to show that for any sufficiently small $\lambda>0$ there exists an appropriate $\Omega$ such that \eqref{eq:D} possesses a radial compact support solution in $\Omega$. See also \textsc{Leach \& Needham} \cite[Section 8.4.4]{needham} for a similar statement in the one-dimensional case.
\end{remark}

\begin{conjecture}
	It is natural to expect that the branches of solutions to \eqref{eq:D} behave as depicted on Figure \ref{fig:caseI}. 
	That is, the problem \eqref{eq:D} has at least two solutions for all $\lambda \in (0,\lambda^*)$ and at least one solution for $\lambda=\lambda^*$. 
	Moreover, we anticipate that the assumption \eqref{eq:caseIas} is technical and can be omitted. 
	Also, we do not know whether solutions to \eqref{eq:D} obtained in Theorems \ref{thm:caseI} and \ref{thm:caseIb} are positive or of compact support type, provided $\lambda>0$.
\end{conjecture}

\section{Case \ref{II}}\label{sec:II}
In this section, we consider the case $0<m<1$ and $m \leq n \leq 1$ which we divide into two subcases.
\subsection{Subcase $0<m=n<1$}
\begin{thm}\label{thm:caseIIa}
	Let $0 < m = n < 1$. Then there exists a solution $u_\lambda$ to \eqref{eq:D} if and only if $\lambda < 1$. Moreover, $u_\lambda$ is positive in $\Omega$ and unique, and the map $\lambda \mapsto u_\lambda$ is a (pointwise) decreasing curve which is smooth in the space $C^1_0(\overline{\Omega})$ and it is of the linearly asymptotically stable type.
\end{thm}
\begin{proof}
	The existence, positivity, and uniqueness of $u_\lambda$ when $\lambda < 1$ are given by Theorems \ref{thm:existence} and \ref{thm:uniq}. The fact that the branch of solutions $\lambda \mapsto u_\lambda$ is decreasing follows from the uniqueness and the method of sub and supersolutions. The smoothness of this branch follows from \cite[Corollary 3.2]{HMV2} giving the linearly asymptotically stable character of solutions as well.
	
	Let $\lambda \geq 1$. Suppose, by contradiction, that \eqref{eq:D} possesses a solution $u$. Arguing as in the proof of Lemma \ref{lem:bound}, we define $M = \max\limits_{x \in \Omega} u(x) = u(x_0) > 0$ for some $x_0 \in \Omega$, and get
	$$
	0 \leq -\Delta u(x_0) = M^m - M^{m+1}  - \lambda M^m \leq 0,
	$$
	which implies that $M=0$, a contradiction.
\end{proof}

The branch of solutions to \eqref{eq:D} obtained in Theorem \ref{thm:caseIIa} is depicted on Figure \ref{fig:caseIIa}. 

\begin{figure}[ht]
	\centering
	\includegraphics[width=0.6\linewidth]{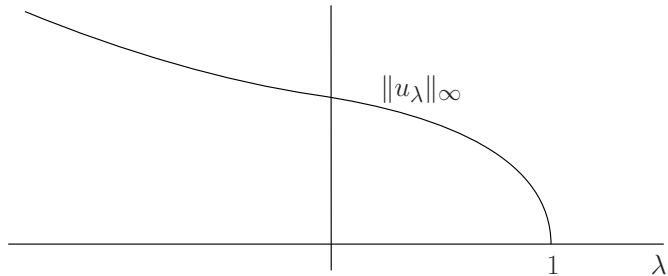}\\
	\caption{The branch of solutions to \eqref{eq:D} in the case $0<m=n<1$.}
	\label{fig:caseIIa}
\end{figure}

\subsection{Subcase $0<m<n\leq 1$}

\begin{thm}\label{thm:caseIIb}
	Let $0<m<n\leq 1$ and $\lambda \in \mathbb{R}$. Then there exists a unique positive solution $u_\lambda$ to \eqref{eq:D}. Moreover, $\lambda \mapsto u_\lambda$ is a (pointwise) decreasing curve which is smooth in the space $C^1_0(\overline{\Omega})$ and it is of the linearly asymptotically stable type.
\end{thm}
\begin{proof}
	As in the proof of Theorem \ref{thm:caseIIa}, the existence, positivity, and uniqueness of $u_\lambda$ are given by Theorems \ref{thm:existence} and \ref{thm:uniq}, the monotonicity of the branch $\lambda \mapsto u_\lambda$ results from the uniqueness and the method of sub and supersolutions, and the smoothness of this branch follows from \cite[Corollary 3.2]{HMV2} which also gives the linearly asymptotically stable character of solutions.
\end{proof}

The branch of solutions to \eqref{eq:D} obtained in Theorem \ref{thm:caseIIb} is depicted on Figure \ref{fig:caseIIb}.

\begin{figure}[ht]
	\centering
	\includegraphics[width=0.6\linewidth]{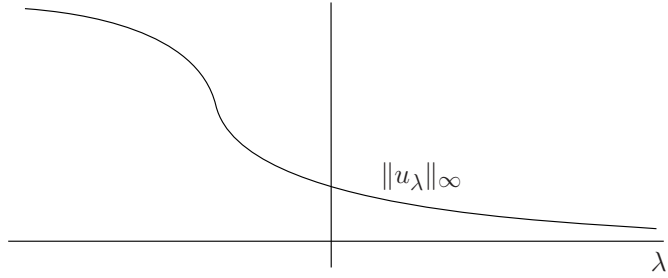}\\
	\caption{The branch of solutions to \eqref{eq:D} in the case $0<m<n\leq 1$.}
	\label{fig:caseIIb}
\end{figure}

\section{Case \ref{III}}\label{sec:III}
In this section, we consider the case $0<m<1 < n < m+1$. 
First, we state the following existence and uniqueness result for \eqref{eq:D}.
\begin{thm}
	Let $0<m<1<n<m+1$ and $\lambda \in \mathbb{R}$. Then there exists a positive solution $u_\lambda$ to \eqref{eq:D}. 
	Moreover, there exists $-\infty < \lambda_0 < 0$ such that $u_\lambda$ is unique for $\lambda \geq \lambda_0$.
\end{thm}
\begin{proof}
	The existence and positivity (for any $\lambda \in \mathbb{R}$) of $u_\lambda$ are given by Theorem \ref{thm:existence}.
	Let us show the uniqueness result. Notice that the uniqueness of $u_\lambda$ for $\lambda \geq 0$
	is given by Theorem \ref{thm:uniq}, and hence we may assume that $\lambda<0$. 
	Arguing as in the proof of Theorem \ref{thm:uniq}, it is sufficient to show that 
	$$
	f(s)-sf'(s) = 
	(1-m)s^m + m s^{m+1} + \lambda (n-1) s^n = s^m \left(1-m+m s + \lambda (n-1) s^{n-m}\right) > 0
	$$
	or, equivalently, 
	\begin{equation}\label{eq:uniq1}
	1-m+m s > |\lambda| (n-1) s^{n-m}
	\end{equation}
	for all $s > 0$. 
	Since $0<n-m<1$, it is clear that the function $g(s) = |\lambda| (n - 1) s^{n-m}$ is increasing and concave and has slope $m$ at the point
	$$
	s^* = \left(\frac{|\lambda| (n-1) (n-m)}{m}\right)^\frac{1}{1+m-n}.
	$$
	Hence, \eqref{eq:uniq1} will be satisfied if
	\begin{equation}\label{eq:uniq2}
	1-m > |\lambda| (n-1) (s^*)^{n-m} - ms^*.
	\end{equation}
	Since $s^* = s^*(\lambda) \to 0$ as $|\lambda| \to 0$, the condition \eqref{eq:uniq2} holds for all sufficiently small $|\lambda|$, and hence the existence of $\lambda_0$ follows.
	Notice that $\lambda_0$ can be estimated explicitly.
\end{proof}

Let us now discuss a multiplicity issue for $\lambda<0$.
\begin{thm}\label{thm:III}
	Let $0<m<1<n<m+1<2^*-1$. Assume also that $|\Omega| < C$, where the constant $C = C(N,m,n)>0$ is specified in \eqref{eq:omegabound} below.
	Then there exist $-\infty<\underline{\lambda}<\overline{\lambda}<0$ such that for any $\lambda \in (\underline{\lambda}, \overline{\lambda})$ the problem \eqref{eq:D} possesses at least three positive solutions.
\end{thm}
\begin{proof}
	Let us outline the idea of the proof. We obtain the existence of three solutions by showing that for each $\lambda < 0$ from a certain interval $(\underline{\lambda}, \overline{\lambda})$, $E_\lambda$ possesses a mountain pass level on the boundary $S_{\rho_\lambda}$ of a ball $B_{\rho_\lambda} = \{u \in W_0^{1,2}(\Omega): \|\nabla u\|_2 \leq \rho_\lambda\}$ such that
	\begin{equation}\label{eq:max<0<inf}
	\max\left\{\inf_{B_{\rho_\lambda}} E_\lambda, \inf_{W_0^{1,2}(\Omega) \setminus B_{\rho_\lambda}} E_\lambda \right\} < 0 < \inf_{S_{\rho_\lambda}} E_\lambda.
	\end{equation}
	Combining this fact with the coercivity  of $E_\lambda$ on $W_0^{1,2}(\Omega)$ (see below), we obtain a global minimum, a local minimum, and a mountain pass critical point of $E_\lambda$. 
	The details are as follows.
	
	First, we show that $E_\lambda$ is coercive on $W_0^{1,2}(\Omega)$ for $\lambda<0$. Indeed, since $m+1<n+1<m+2$, we apply H\"older's inequality to get
	\begin{align*}
	E_\lambda(u) 
	\geq
	&\frac{1}{2} \int_\Omega |\nabla u|^2 \, dx 
	\\
	&-
	C_1 \left(\int_\Omega |u^+|^{m+2} \, dx\right)^\frac{m+1}{m+2}
	+
	\frac{1}{m+2} \int_\Omega |u^+|^{m+2} \, dx  
	+
	\lambda C_2 \left(\int_\Omega |u^+|^{m+2} \, dx\right)^\frac{n+1}{m+2},
	\end{align*}
	where $C_1,C_2>0$ are independent of $u \in W_0^{1,2}(\Omega)$. 
	Since the function $h(s) = -C_1 s^\frac{m+1}{m+2} + \frac{1}{m+2} s + \lambda C_2 s^\frac{n+1}{m+2}$ is bounded from below on $[0,+\infty)$, we easily deduce that $E_\lambda(u) \to +\infty$ as $\|\nabla u\|_2 \to +\infty$, which is the desired coercivity.
	
	Second, we show that there exists $\underline{\lambda}<0$ such that for any $\lambda \in (\underline{\lambda},0)$ there are $\rho_\lambda>0$ and $\alpha_\lambda>0$ satisfying $E_\lambda(u) \geq \alpha_\lambda$ provided $\|\nabla u\|_2 = \rho_\lambda$.
	Indeed, consider $S_1 = \{u \in W_0^{1,2}(\Omega): \|\nabla u\|_2 = 1\}$ and take any $u \in S_1$. Recalling that $\lambda<0$, we have
	\begin{equation}\label{eq:Ilower}
	E_\lambda(t u) \geq 
	t^{m+1} 
	\left(
	\frac{t^{1-m}}{2}
	-
	A_1
	+
	\lambda t^{n-m} A_2
	\right),
	\end{equation}
	where 
	$$
	A_1 = \frac{1}{m+1}\sup_{v \in S_1} \int_\Omega |v^+|^{m+1} \, dx > 0
	\quad \text{and} \quad 
	A_2 = \frac{1}{n+1} \sup_{v \in S_1} \int_\Omega |v^+|^{n+1} \, dx > 0.
	$$
	Notice that $A_1$ and $A_2$ are achieved, due to the Rellich-Kondrachov theorem. 
	Let us analyze the function $g_\lambda(s) = \frac{s^{1-m}}{2} - A_1 +  \lambda s^{n-m} A_2$ for $s>0$. 
	Recalling that $0<1-m<n-m$, we see that $\frac{s^{1-m}}{2} - A_1$ is the leading term as $s \to 0$ and $\lambda s^{n-m} A_2$ is the leading term as $s \to +\infty$. 
	Therefore, looking at the first derivative of $g_\lambda(s)$, we deduce that $g_\lambda(s)$ has exactly one critical point which is the point of global maximum. We denote this point as $s_\lambda$ and note that 
	$$
	s_{\lambda} = \left(\frac{1-m}{-2 \lambda (n-m)A_2}\right)^\frac{1}{n-1}.
	$$
	Moreover, straightforward analyzis implies that the map $\lambda \mapsto g_\lambda(s_\lambda)$, $\lambda \in (-\infty,0)$, increases with respect to $\lambda$, it tends to $-A_1$ as $\lambda \to -\infty$, and it tends to $+\infty$ as $\lambda \to 0$. Therefore, there exists $\underline{\lambda}<0$ such that $g_\lambda(s_\lambda)>0$ for all $\lambda \in (\underline{\lambda},0)$ and $g_{\underline{\lambda}}(s_{\underline{\lambda}}) = 0$. 
	Directly investigating the function $g_\lambda(s_\lambda)$, we find that
	$$
	\underline{\lambda} = 
	-
	\frac{(n-1)^\frac{n-1}{1-m} (1-m)}{2^{\frac{n-m}{1-m}} (n-m)^\frac{n-m}{1-m} A_1^{\frac{n-1}{1-m}} A_2}. 
	$$
	Putting $\rho_\lambda = s_\lambda$, $\alpha_\lambda = s_\lambda^{m+1}g_\lambda(s_\lambda)>0$, and recalling that $u \in S_1$ was arbitrary, we conclude from \eqref{eq:Ilower}  that $E_\lambda(\rho_\lambda u) \geq \alpha_\lambda$, which completes the proof of the desired claim. 
	
	Third, we show the existence of a constant $C>0$ which depends only on $N$, $m$, $n$ such that if $|\Omega|<C$, then there is $\overline{\lambda} \in (\underline{\lambda}, 0)$ such that for any $\lambda \in (\underline{\lambda}, \overline{\lambda})$ there exists $v$ satisfying $\|\nabla v\|_2 > \rho_\lambda$ and $E_\lambda(v) < 0$.
	Let $w \in S_1$, $w \geq 0$, be a maximizer for $A_2$. 
	We estimate $E_\lambda(t w)$ from above as follows:
	$$
	E_\lambda(tw)
	\leq  
	t^2 \left(
	\frac{1}{2} 
	+
	t^{m} A_3 
	+
	\lambda t^{n-1} A_2
	\right),
	$$
	where 
	$$
	A_3 = \frac{1}{m+2}\sup_{v \in S_1} \int_{\Omega} |v^+|^{m+2} \, dx > 0.
	$$
	Arguing in much the same way as in the previous step, let us investigate the function $f_\lambda(s) = \frac{1}{2} + s^{m} A_3 + \lambda s^{n-1} A_2$ for $s>0$. 
	Thanks to the assumption $0<n-1<m$, the term $\lambda s^{n-1} A_2$ is leading as $s \to 0$ and the term $s^{m} A_3$ is leading as $s \to +\infty$. This implies that $f_\lambda(s)$ has exactly one critical point $\hat{s}_\lambda$ which is the point of global minimum. Notice that
	$$
	\hat{s}_{\lambda} = \left(\frac{-\lambda (n-1) A_2}{mA_3}\right)^\frac{1}{m+1-n}.
	$$
	Moreover, the map $\lambda \mapsto f_\lambda(\hat{s}_\lambda)$, $\lambda \in (-\infty,0)$, increases with respect to $\lambda$, it tends to $-\infty$ as $\lambda \to -\infty$, and it tends to $\frac{1}{2}$ as $\lambda \to 0$. Thus, one can find $\overline{\lambda}<0$ such that $f_\lambda(\hat{s}_\lambda)<0$ for all $\lambda < \overline{\lambda}$ and $f_{\overline{\lambda}}(\hat{s}_{\overline{\lambda}}) = 0$. 
	Performing direct calculations, we deduce that
	$$
	\overline{\lambda} = 
	-
	\frac{m A_3^\frac{n-1}{m}}{2^\frac{m+1-n}{m} (n-1)^\frac{n-1}{m} (m+1-n)^\frac{m+1-n}{m} A_2}.
	$$
	Comparing now $\underline{\lambda}$ with $\overline{\lambda}$ and $s_{\underline{\lambda}}$ with $\hat{s}_{\overline{\lambda}}$, we see that if
	\begin{equation}\label{eq:a1a3}
	A_1^\frac{1}{1-m} A_3^\frac{1}{m} < \widetilde{C},
	\end{equation}
	then $\underline{\lambda} < \overline{\lambda}$ and $s_{\underline{\lambda}} < \hat{s}_{\overline{\lambda}}$.
	Here the constant $\widetilde{C}>0$ depends only on $m$ and $n$ and has the following explicit expression:
	$$
	\widetilde{C} = \left(\frac{n-1}{2}\right)^\frac{1}{m(1-m)}\left(\frac{1-m}{m}\right)^\frac{1}{n-1} \frac{(m+1-n)^\frac{m+1-n}{m(n-1)}}{(n-m)^\frac{n-m}{(1-m)(n-1)}}.
	$$
	Let us show that \eqref{eq:a1a3} is satisfied provided the measure of $\Omega$ is sufficiently small. 
	Using the Faber-Krahn inequality, we obtain that
	\begin{align*}
	A_1 
	&\equiv
	\frac{1}{m+1} 
	\sup_{v \in W_0^{1,2}(\Omega) \setminus \{0\}} \frac{\int_\Omega |v|^{m+1}\, dx}{\left(\int_\Omega |\nabla v|^2\, dx\right)^\frac{m+1}{2}}
	\\
	&\leq 
	\frac{|\Omega|^{\frac{m+1}{N}+\frac{1-m}{2}}}{(m+1) |\mathcal{B}_1|^{\frac{m+1}{N}+\frac{1-m}{2}}}
	\sup_{v \in W_0^{1,2}(\mathcal{B}_1) \setminus \{0\}} \frac{\int_{\mathcal{B}_1} |v|^{m+1}\, dx}{\left(\int_{\mathcal{B}_1} |\nabla v|^2\, dx\right)^\frac{m+1}{2}}
	=:
	\frac{|\Omega|^{\frac{m+1}{N}+\frac{1-m}{2}}}{|\mathcal{B}_1|^{\frac{m+1}{N}+\frac{1-m}{2}}} A_1^*
	\end{align*}
	and
	\begin{align*}
	A_3 
	&\equiv 
	\frac{1}{m+2}
	\sup_{v \in W_0^{1,2}(\Omega) \setminus \{0\}} \frac{\int_\Omega |v|^{m+2}\, dx}{\left(\int_\Omega |\nabla v|^2\, dx\right)^\frac{m+2}{2}}
	\\
	&\leq 
	\frac{|\Omega|^{\frac{m+2}{N}-\frac{m}{2}}}{(m+2)|\mathcal{B}_1|^{\frac{m+2}{N}-\frac{m}{2}}} 
	\sup_{v \in W_0^{1,2}(\mathcal{B}_1) \setminus \{0\}} \frac{\int_{\mathcal{B}_1} |v|^{m+2}\, dx}{\left(\int_{\mathcal{B}_1} |\nabla v|^2\, dx\right)^\frac{m+2}{2}}
	=:
	\frac{|\Omega|^{\frac{m+2}{N}-\frac{m}{2}}}{|\mathcal{B}_1|^{\frac{m+2}{N}-\frac{m}{2}}} A_3^*,
	\end{align*}
	where $\mathcal{B}_1$ is a unit ball in $\mathbb{R}^N$.
	Therefore, we deduce that if
	\begin{equation}\label{eq:omegabound}
	|\Omega| 
	< C 
	:=
	\widetilde{C}^\frac{Nm(1-m)}{2} |\mathcal{B}_1| (A_1^*)^\frac{-Nm}{2} (A_3^*)^\frac{-N(1-m)}{2},
	\end{equation}
	then \eqref{eq:a1a3} holds true. 
	Since \eqref{eq:omegabound} is supposed to hold, we have $\underline{\lambda} < \overline{\lambda}$ and $s_{\underline{\lambda}} < \hat{s}_{\overline{\lambda}}$.
	
	Let us show that $s_\lambda < \hat{s}_{\lambda}$ for any  $\lambda \in (\underline{\lambda}, \overline{\lambda})$. 
	Since $\hat{s}_{\lambda}$ is decreasing with respect to $\lambda$, we get $\hat{s}_{\overline{\lambda}} < \hat{s}_{\lambda}$ for any $\lambda < \overline{\lambda}$.
	At the same time, since $s_{\lambda}$ increases with respect to $\lambda$, we have $s_{\underline{\lambda}} < s_\lambda$ for any $\lambda > \underline{\lambda}$. 
	If we suppose that $s_{\tilde{\lambda}} = \hat{s}_{\tilde{\lambda}}$ for some $\tilde{\lambda} \in (\underline{\lambda}, \overline{\lambda})$, then we get a contradiction since 
	$s^{m+1}g_\lambda(s) \leq E_\lambda(s w) \leq s^2f_\lambda(s)$ for any $s>0$, but $g_{\tilde{\lambda}}(s_{\tilde{\lambda}}) > 0$ and $f_{\tilde{\lambda}}(s_{\tilde{\lambda}})<0$.
	Therefore, recalling that we put $\rho_\lambda = s_\lambda$, we conclude that $\hat{s}_\lambda > \rho_\lambda$ and 
	$$
	E_\lambda(\hat{s}_\lambda w) \leq \hat{s}_\lambda^2 f_\lambda(\hat{s}_\lambda) < 0
	$$
	for any $\lambda \in (\underline{\lambda}, \overline{\lambda})$, whenever \eqref{eq:omegabound} holds true.
	
	Finally, taking any nonnegative $u \in W_0^{1,2}(\Omega) \setminus \{0\}$ and analyzing the fibers $\phi_u(t)$ associated with $E_\lambda(u)$, we easily see that $E_\lambda(tu)<0$ for every $\lambda \in \mathbb{R}$ and all sufficiently small $t>0$, since $-\frac{t^{m+1}}{m+1} \int_\Omega u^{m+1} \,dx$ is the leading term as $t \to 0$. 
	In particular, we get $\inf_{B_{\rho_\lambda}} E_\lambda < 0$ for any $\lambda \in (\underline{\lambda}, \overline{\lambda})$.

	We are ready to obtain three critical points of $E_\lambda$. 
	Recalling \eqref{eq:max<0<inf}, the first critical point $u_\lambda$ comes as a minimizer of $E_\lambda$ over the ball $B_{\rho_\lambda}$. We have $\|\nabla u_\lambda\|_2 < \rho_\lambda$ and $E_\lambda(u_\lambda)<0$. 
	Recalling that $E_\lambda$ is coercive on $W_0^{1,2}(\Omega)$, we conclude that $E_\lambda$ satisfies the Palais-Smale condition. 
	Therefore, the second critical point $v_\lambda$ is obtained via the standard mountain pass theorem, see, e.g., \cite{rabinowitz}. We have $E_\lambda(v_\lambda)>0$.
	To obtain the third critical point $w_\lambda$, let us minimize $E_\lambda$ over $W_0^{1,2}(\Omega) \setminus B_{\rho_\lambda}$. 
	Since $E_\lambda$ is coercive on $W_0^{1,2}(\Omega)$ and in view of \eqref{eq:max<0<inf}, we can apply Ekeland's variational principle (see \cite[Theorem 3.1]{ekeland}) and obtain a minimizing sequence $\{w_n\} \subset W_0^{1,2}(\Omega) \setminus B_{\rho_\lambda}$ for $E_\lambda$ which is at the same time a Palais-Smale sequence for $E_\lambda$. Since $E_\lambda$ satisfies the Palais-Smale condition on $W_0^{1,2}(\Omega)$, we conclude that the third critical point $w_\lambda$ exists, $\|\nabla w_\lambda\|_2 > \rho_\lambda$, and $E_\lambda(w_\lambda)<0$.
	Clearly, $u_\lambda$, $v_\lambda$, and $w_\lambda$ are different from each other by construction, and positive in $\Omega$ due to Lemma \ref{lem:smp}.
\end{proof}

\begin{figure}[ht]
	\centering
	\includegraphics[width=0.6\linewidth]{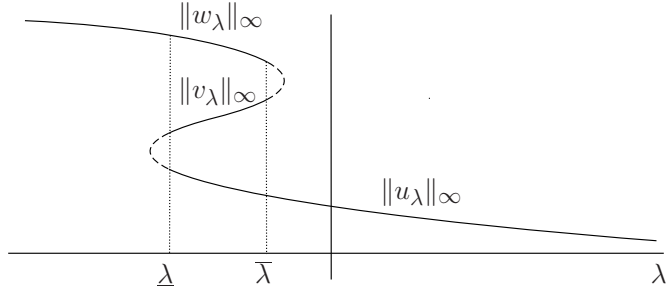}\\
	\caption{The branches of solutions to \eqref{eq:D} in the case $0<m<1<n<m+1$.}
	\label{fig:caseIIIa}
\end{figure}

\begin{remark}
	A result on the existence of three solutions was obtained by \textsc{Lubyshev} in \cite{lubyshev} for a problem similar to \eqref{eq:D} under the case \ref{III}, however with the parameter $\lambda$ placed in front of the difference $u^m - u^{m+1}$ instead of the term $u^{n}$ as in our case. 
	Although the fibers of our functional $E_\lambda$ and of the functional considered in \cite{lubyshev} have a similar ``three critical points''-structure, their parametric behaviour is different. This fact prevents a direct application of the result of \cite{lubyshev} to our problem. 
\end{remark}

\begin{remark}
	Results concerning the existence of $S$-shape type bifurcation curves have been also obtained by using different approaches. For instance, \textsc{Crandall \& Rabinowitz} used continuation methods in \cite{CR},  \textsc{Brown, Ibrahim \& Shivaji} used a tricky application of sub and supersolutions in \cite{BIS},
	\textsc{D\'iaz, Hern\'andez \& Tello} used a combination of the sub and supersolution method with the topological index theory in \cite{DHT},
	and
	\textsc{Arcoya, D\'iaz \& Tello} used a global bifurcation approach in \cite{arcoya}.
\end{remark}

\begin{conjecture}
	We anticipate that the interval $(\underline{\lambda}, \overline{\lambda})$ obtained in Theorem \ref{thm:III} can be expanded in such a way that \eqref{eq:D} possesses at least three solutions for all $\lambda \in (\underline{\lambda}, \overline{\lambda})$ and at least two solutions for $\lambda=\underline{\lambda}$ and $\lambda = \overline{\lambda}$, i.e., the branch of solutions to \eqref{eq:D} has an $S$-shape, see Figure \ref{fig:caseIIIa}.
	Moreover, we presume that the assumption on the size of $\Omega$ used in Theorem \ref{thm:III} is technical and can be omitted.
\end{conjecture}

\section{Case \ref{IV}}\label{sec:IV}
In this section, we study the problem \eqref{eq:D} under the assumption $0<m<1 <m+1 \leq n$.
We consider two subcases.

\subsection{Subcase $n=m+1$}
First, we state the following existence result for \eqref{eq:D} when $\lambda \geq -1$.
\begin{thm}
	Let $0<m<1<n=m+1$ and $\lambda \geq -1$. Then there exists a unique positive solution $u_\lambda$ to \eqref{eq:D}.
\end{thm}
\begin{proof}
	The existence, positivity, and uniqueness of $u_\lambda$ in the case $\lambda > -1$ are given by Theorems \ref{thm:existence} and \ref{thm:uniq}. 
	Assume that $\lambda=-1$. Then the problem \eqref{eq:D} turns to be
	\begin{equation}\label{eq:Dm}
	\left\{
	\begin{aligned}
	-\Delta u &= u^m &&\text{in}~ \Omega,\\
	u &= 0 &&\text{on}~ \partial\Omega.
	\end{aligned}
	\right.
	\end{equation}
	Since $0<m<1$, \eqref{eq:Dm} possesses a unique positive solution, see, e.g., \cite{brezisoswald}. The proof is complete.
\end{proof}

Let us now discuss the case $\lambda < -1$. 
We rewrite the problem \eqref{eq:D} in the form 
\begin{equation}\label{eq:Dnm}
\left\{
\begin{aligned}
-\Delta u &= u^m + (-\lambda-1) u^{m+1}&&\text{in}~ \Omega,\\
u &\geq 0,~ u \not\equiv 0 &&\text{in}~ \Omega,\\
u &= 0 &&\text{on}~ \partial\Omega.
\end{aligned}
\right.
\end{equation}
Notice that $u$ is a positive solution to \eqref{eq:Dnm} if and only if $v = (-\lambda-1)^\frac{1}{m} u$ is a solution to the problem
\begin{equation}\label{eq:ABC}
\left\{
\begin{aligned}
-\Delta v &= (-\lambda-1)^\frac{1-m}{m} v^m + v^{m+1}&&\text{in}~ \Omega,\\
v &> 0 &&\text{in}~ \Omega,\\
v &= 0 &&\text{on}~ \partial\Omega.
\end{aligned}
\right.
\end{equation}
Denoting $\mu = (-\lambda-1)^\frac{1-m}{m}$, we see that \eqref{eq:ABC} is the standard problem with the convex-concave nonlinearity, and the description of the existence of solutions to \eqref{eq:ABC} is given by \textsc{Ambrosetti, Brezis \& Cerami} in \cite{ABC}.
Translating this description to the problem \eqref{eq:D}, we deduce that there exists $\lambda^*<-1$ such that for any $\lambda \in (\lambda^*,-1)$, \eqref{eq:D} possesses a positive solution $u_\lambda$. 
If $\lambda=\lambda^*$, then \eqref{eq:D} possesses a positive weak solution in $X$. 
For any $\lambda<\lambda^*$, \eqref{eq:D} does not have solutions.
Moreover, if $m < 2^*-2$, then \eqref{eq:D} possesses another positive solution $v_\lambda$ for any $\lambda \in (\lambda^*,-1)$. 
We refer to Figure \ref{fig:caseIIIb}, where the behaviour of the corresponding branches is depicted.

Let us provide an explicit lower bound for $\lambda^*$.
\begin{lemma}\label{lem:IIIb}
	Let $0<m<1<n=m+1$ and 
	$$
	\lambda < -1 - \lambda_1^\frac{1}{1-m} (1-m) m^\frac{m}{1-m}.
	$$
	Then \eqref{eq:D} does not possess any solution.
\end{lemma}
\begin{proof}
	Assume that \eqref{eq:D} possesses a solution $u_\lambda$ for some $\lambda < -1$. By Lemma \ref{lem:smp}, this solution is positive in $\Omega$.
	Multiplying the equation \eqref{eq:D} by $\varphi_1$ and integrating by parts, we obtain
	$$
	\lambda_1 \int_\Omega u_\lambda \varphi_1 \,dx = \int_\Omega u_\lambda^m \varphi_1 \, dx- (\lambda+1) \int_\Omega u_\lambda^{m+1} \varphi_1 \,dx,
	$$
	or, equivalently,
	\begin{equation}\label{eq:lem:contr1}
	\int_\Omega \left(\lambda_1 u_\lambda^{1-m}  + (\lambda+1) u_\lambda - 1 \right) u_\lambda^m \varphi_1 \,dx = 0.
	\end{equation}
	Let us investigate the function $h_\lambda(s) = \lambda_1 s^{1-m} + (\lambda+1) s - 1$ for $s>0$.
	Since $0<m<1$, the term $\lambda_1 s^{1-m}-1$ is leading as $s \to 0$, and the term $(\lambda+1) s$ is leading as $s \to +\infty$. 
	Moreover, $h_\lambda''(s)<0$ for $s>0$.
	Therefore, recalling that $\lambda<-1$, we see that $h_\lambda(s)$ has exactly one critical point which is the point of global maximum. Let us denote this point as $s_\lambda$. 
	Arguing in a straightforward way, we deduce that $h_\lambda(s_\lambda)$ increases with respect to $\lambda \in (-\infty, -1)$.
	Thus, if $h_{\overline{\lambda}}(s_{\overline{\lambda}}) = 0$ for some $\overline{\lambda}$, then $h_\lambda(s_\lambda) < 0$ for all $\lambda < \overline{\lambda}$, and hence for such $\lambda$ we get a contradiction with \eqref{eq:lem:contr1}. 
	Analyzing directly the function $h_\lambda(s_\lambda)$, we deduce that 
	$$
	\overline{\lambda} = -1 - \lambda_1^\frac{1}{1-m} (1-m) m^\frac{m}{1-m},
	$$
	which finishes the proof.
\end{proof}

\begin{figure}[ht]
	\centering
	\includegraphics[width=0.6\linewidth]{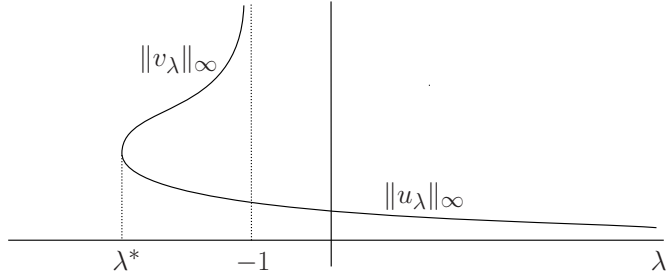}\\
	\caption{The branches of solutions to \eqref{eq:D} in the case $0<m<1<n=m+1$.}
	\label{fig:caseIIIb}
\end{figure}

\subsection{Subcase $1<m+1<n$}\label{sec:IVa}

In this case, the problem \eqref{eq:D} has essentially the same behaviour as the convex-concave problem \eqref{eq:ABC} from the previous subsection. 
Indeed, for $\lambda<0$, the problem \eqref{eq:D} is a special case of a more general problem with a sublinear behaviour near $0$ and superlinear behaviour near $+\infty$, studied by \textsc{De Figueiredo, Gossez \& Ubilla} in \cite{DGU1, DGU2}.
Applying the results of \cite{DGU2} to \eqref{eq:D}, we deduce that there exists $\lambda^* \in (-\infty,0)$ such that \eqref{eq:D} possesses a positive solution $u_\lambda$, satisfying $E_\lambda(u_\lambda)<0$, for any $\lambda \in (\lambda^*,0)$, and \eqref{eq:D} does not have solutions for any $\lambda<\lambda^*$. 
Moreover, if $n < 2^*-1$, then \eqref{eq:D} possesses another positive solution $v_\lambda$ for any $\lambda \in [\lambda^*,0)$, and $E_\lambda(v_\lambda) \leq 0$ for $\lambda=\lambda^*$.
Furthermore, an explicit upper estimate for $\lambda^*$ is provided by \cite[Theorem 2.1]{DGU1}.

Let us complement these facts noting that \eqref{eq:D} possesses a unique positive solution $u_\lambda$ for any $\lambda \geq 0$, which follows from Theorems \ref{thm:existence} and \ref{thm:uniq}.
Figure \ref{fig:caseIV} depicts the behaviour of the corresponding branches. 

\begin{figure}[ht]
	\centering
	\includegraphics[width=0.6\linewidth]{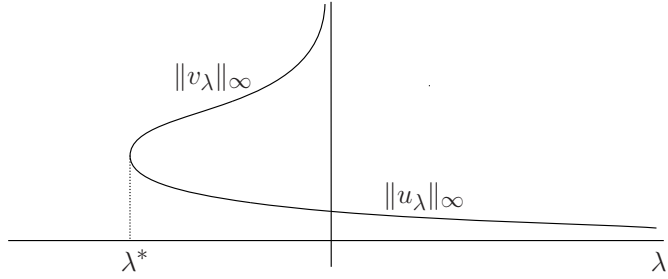}\\
	\caption{The branches of solutions to \eqref{eq:D} in the case $0 < m < 1 < m+1 < n$.}
	\label{fig:caseIV}
\end{figure}

\section{Case \ref{VI}}\label{sec:VI}
In this section, we assume that $m=1$ and $0<n \leq 1$.  
We divide these assumptions into two subcases.

\subsection{Subcase $0<n<m=1$}
Let us consider the following more general version of the problem \eqref{eq:D} under the imposed assumptions on $m$ and $n$:
\begin{equation}\label{eq:Dsign}
\left\{
\begin{aligned}
-\Delta u &= - \lambda |u|^{n-1} u + u - |u| u  &&\text{in}~ \Omega,\\
u &= 0 &&\text{on}~ \partial\Omega.
\end{aligned}
\right.
\end{equation}
Note that any nonzero and nonnegative solution to \eqref{eq:Dsign} is a solution to \eqref{eq:D}.
The problem \eqref{eq:Dsign} with $\lambda>0$ is a special case of a problem studied by \textsc{Perera} in \cite{perera}. It was proved in \cite{perera} that if $\lambda_1<1$, then there exists $\lambda_*>0$ such that \eqref{eq:Dsign} possesses at least two nonzero nonnegative solutions $u_\lambda$ and $v_\lambda$ for any $\lambda \in (0, \lambda_*)$, where $u_\lambda$ is a global minimizer of $E_\lambda$ and $v_\lambda$ is a critical point of the mountain pass type.
Therefore, \eqref{eq:D} has at least two solutions. 
Let us complement this fact by the following two results.

\begin{thm}\label{thm:VIa}
	Let $m=1$ and $0<n<1$.
	If $\lambda<0$, then there is a unique positive solution $u_\lambda$ to \eqref{eq:D}.
	If $\lambda = 0$, then \eqref{eq:D} possesses a solution $u_0$ if and only if $\lambda_1 < 1$. 
	Moreover, $u_0$ is positive in $\Omega$ and unique.
	Furthermore, if \eqref{eq:D} possesses a solution $u_\lambda$ for $\lambda>0$, then $\lambda_1<1$.
\end{thm}
\begin{proof}
	The existence, positivity, and uniqueness of $u_\lambda$ for $\lambda< 0$ are given by Theorems \ref{thm:existence} and \ref{thm:uniq}. 
	Let us consider the existence for $\lambda = 0$. 
	In this case, the problem \eqref{eq:D} turns to be
	\begin{equation}\label{eq:Dm=1}
	\left\{
	\begin{aligned}
	-\Delta u &= u-u^2 &&\text{in}~ \Omega,\\
 	u &\geq 0, ~ u \not\equiv 0 &&\text{in}~ \Omega,\\
	u &= 0 &&\text{on}~ \partial\Omega,
	\end{aligned}
	\right.
	\end{equation}
	which is the problem with the logistic nonlinearity, and it is known that \eqref{eq:Dm=1} possesses a (unique positive) solution if and only if $\lambda_1 < 1$, see, e.g., \cite{brezisoswald}.
	The last assertion of the theorem is given by Theorem \ref{thm:existence2}.
\end{proof}

\begin{figure}[ht]
	\centering
	\includegraphics[width=0.6\linewidth]{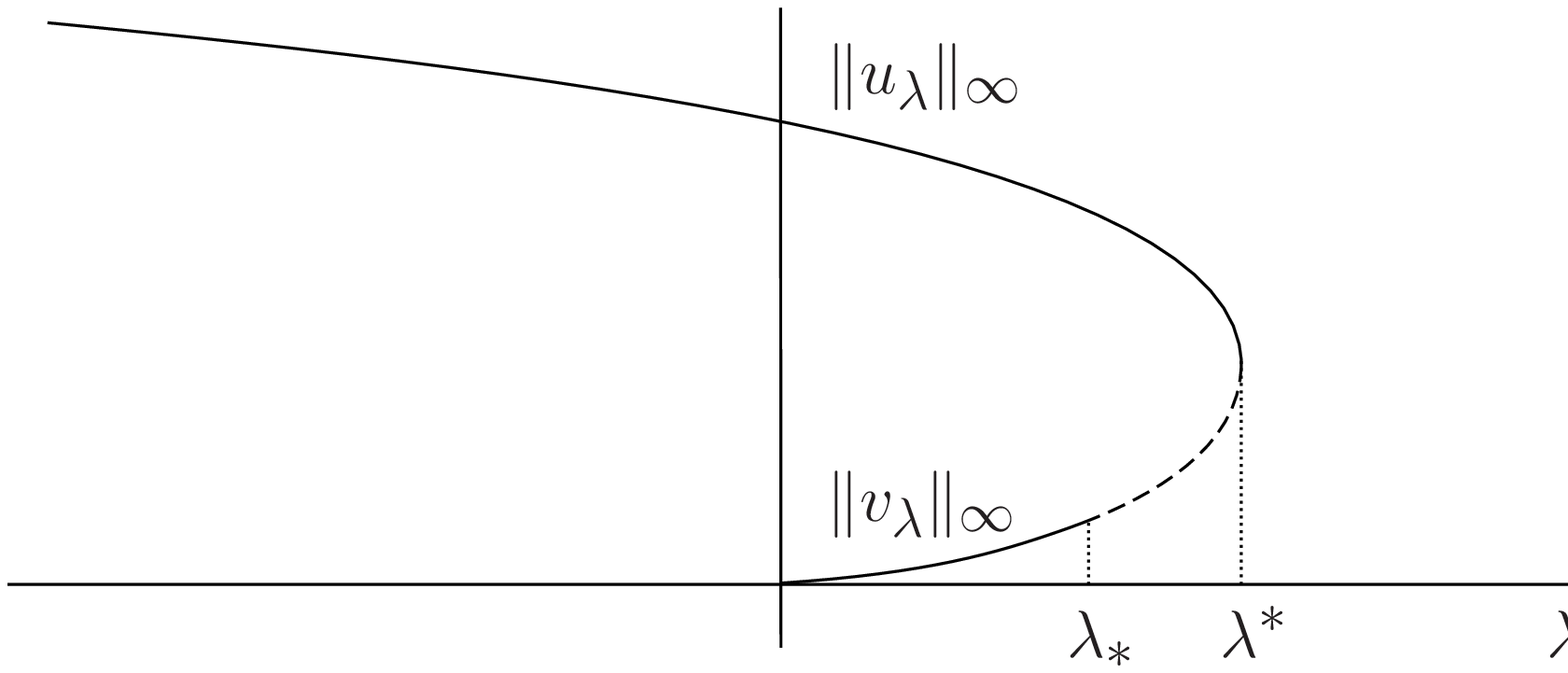}\\
	\caption{The branches of solutions to \eqref{eq:D} in the case $0<n<m=1$ and $\lambda_1<1$.}
	\label{fig:caseVI}
\end{figure}

\begin{thm}\label{thm:VI1}
	Let $m=1$, $0<n<1$, and $\lambda_1<1$. 
	Then there exists $\lambda^*>0$ satisfying
	$$
	\lambda^* \leq \frac{(1-\lambda_1)^{2-n}(1-n)^{1-n}}{(2-n)^{2-n}},
	$$	
	such that \eqref{eq:D} possesses a solution $u_\lambda$ for all $\lambda \in (0, \lambda^*)$, and \eqref{eq:D} has no solution for any $\lambda > \lambda^*$. 
\end{thm}
\begin{proof}
	Recalling the result of \cite{perera} cited above, let us define
	$$
	\lambda^* = \sup\{\lambda>0:~ \eqref{eq:D} ~\text{has a solution}\}.
	$$
	Clearly, $0<\lambda_* \leq \lambda^*$. 
	To obtain the upper bound for $\lambda^*$ we argue in much the same way as in the proof of Lemma \ref{lem:IIIb}. 
	Let $u_\lambda$ be a solution to \eqref{eq:D}. 
	Multiplying the equation \eqref{eq:D} by $\varphi_1$ and integrating by parts, we get
	$$
	\lambda_1 \int_\Omega u_\lambda \varphi_1 \,dx = \int_\Omega u_\lambda \varphi_1 \, dx -  \int_\Omega u_\lambda^{2} \varphi_1 \,dx - \lambda \int_\Omega u_\lambda^{n} \varphi_1 \,dx,
	$$
	or, equivalently,
	\begin{equation}\label{eq:lem:contr3}
	\int_\Omega \left((\lambda_1-1) u_\lambda^{1-n}  + u_\lambda^{2-n} + \lambda \right) u_\lambda^n \varphi_1 \,dx = 0.
	\end{equation}
	Let us investigate the function $h_\lambda(s) = (\lambda_1-1) s^{1-n} + s^{2-n} + \lambda$ for $s>0$.
	Since $0<n<1$, we see that $(\lambda_1-1) s^{1-n}+\lambda$ is the leading term as $s \to 0$, and $s^{2-n}$ is the leading term as $s \to +\infty$. 
	Moreover, recalling that $\lambda_1<1$, we have $h_\lambda''(s)>0$ for all $s>0$.
	Therefore, $h_\lambda(s)$ has exactly one critical point in $(0,+\infty)$ which is the point of global minimum. Let us denote this point as $s_\lambda$ and note that $h_\lambda(s_\lambda)$ increases with respect to $\lambda$.
	Thus, if $h_{\overline{\lambda}}(s_{\overline{\lambda}}) = 0$ for some $\overline{\lambda}$, then $h_\lambda(s_\lambda) > 0$ for all $\lambda > \overline{\lambda}$, and hence for such $\lambda$ we get a contradiction with \eqref{eq:lem:contr3}. 
	Directly analyzing the function $h_\lambda(s)$, we obtain that 
	$$
	\overline{\lambda} = \frac{(1-\lambda_1)^{2-n}(1-n)^{1-n}}{(2-n)^{2-n}},
	$$
	which finishes the proof.
\end{proof}
	
\begin{remark}\label{rem:caseV:compact}
	As in the case \ref{I}, the results of \cite{FLS} or \cite{GST} can be applied to show that for any sufficiently small $\lambda>0$ there exists an appropriate $\Omega$ such that \eqref{eq:D} possesses a radial compact support solution in $\Omega$.
\end{remark}	
	
\begin{conjecture}
	We anticipate that the branches of solutions to \eqref{eq:D} behave as depicted on Figure \ref{fig:caseVI}. 
	That is, the problem \eqref{eq:D} has at least two solutions for all $\lambda \in (0, \lambda^*)$ and at least one solution for $\lambda=\lambda^*$. 
	Also, it is an open question weather solutions to \eqref{eq:D} with $\lambda>0$ obtained in \cite{perera} are positive or of compact support type.
\end{conjecture}

\subsection{Subcase $n=m=1$}
Under the assumption $n=m=1$, the problem \eqref{eq:D} can be written as
\begin{equation}\label{eq:Dmn=1}
\left\{
\begin{aligned}
-\Delta u &= (1-\lambda) u-u^2 &&\text{in}~ \Omega,\\
 u &\geq 0, ~ u \not\equiv 0 &&\text{in}~ \Omega,\\
u &= 0 &&\text{on}~ \partial\Omega,
\end{aligned}
\right.
\end{equation}
which is the problem with the logistic nonlinearity, and it is known that \eqref{eq:Dmn=1} possesses a (unique positive) solution if and only if $\lambda < 1 - \lambda_1$, see, e.g., \cite{brezisoswald}. 
The behaviour of the branch of solutions is reminiscent of whose depicted on Figure \ref{fig:caseIIa}.

\section{Case \ref{VII}}\label{sec:VII}
In this section, we study the problem \eqref{eq:D} under the assumption $m=1$ and $1 < n \leq 2$. 
Let us split this assumption into two subcases.

\subsection{Subcase $1<n<2$}
\begin{thm}\label{thm:VII1}
	Let $m=1 < n < \min\{2,2^*-1\}$. Then the following assertions hold:
	\begin{enumerate}[label={\rm(\roman*)}]
		\item\label{thm:VII2-1} If $\lambda_1<1$, then \eqref{eq:D} possesses a positive solution $u_\lambda$ for any $\lambda \in \mathbb{R}$. 
		Moreover, if $\lambda \geq 0$, then $u_\lambda$ is unique.
		\item\label{thm:VII2-2} If $\lambda_1 = 1$, then \eqref{eq:D} possesses a positive solution $u_\lambda$ if and only if $\lambda<0$.
	\end{enumerate}
\end{thm}
\begin{proof}
	\ref{thm:VII2-1} 
	The existence and  positivity of $u_\lambda$ are given by Theorem \ref{thm:existence2}. The uniqueness of $u_\lambda$ for $\lambda \geq 0$ is given by Theorem \ref{thm:uniq}. 
	The assertion \ref{thm:VII2-2} is contained in  Theorem \ref{thm:existence2}.
\end{proof}

\begin{figure}[ht]
	\centering
	\includegraphics[width=0.6\linewidth]{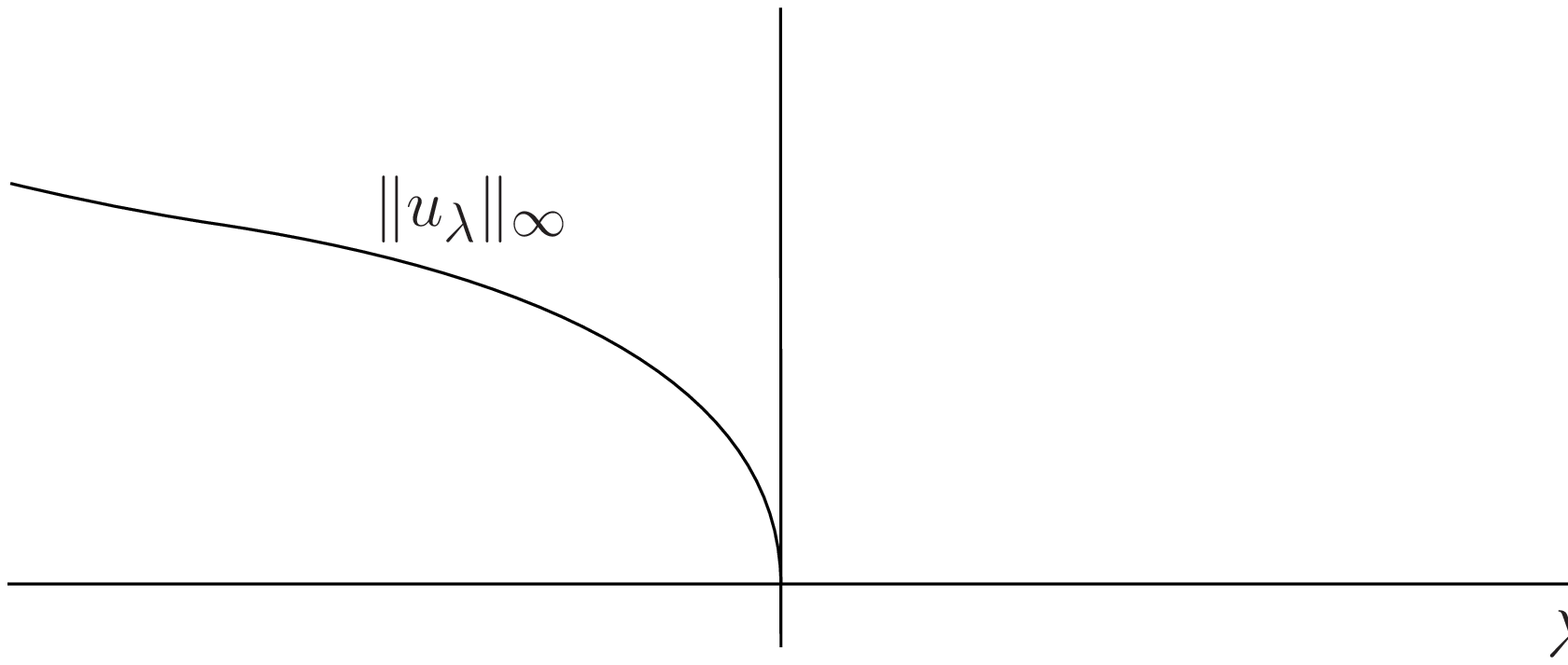}\\
	\caption{The branch of solutions to \eqref{eq:D} in the case $m=1 < n < \min\{2,2^*-1\}$ and $\lambda_1=1$.}
	\label{fig:caseVIIb}
\end{figure}

\begin{thm}\label{thm:VIIb}
	Let $m=1 < n < \min\{2,2^*-1\}$ and $\lambda_1 > 1$. 
	Then there exists $\lambda^*$ satisfying
	$$
	\lambda^* \leq -\frac{(\lambda_1-1)^{2-n}}{(n-1)^{n-1}(2-n)^{2-n}} < 0,
	$$
	such that \eqref{eq:D} possesses a positive solution $u_\lambda$ for any $\lambda < \lambda^*$, and \eqref{eq:D} has no solution for any $\lambda > \lambda^*$. 
	If, moreover, $N < 6$, then there exists $\lambda_* \in (-\infty, \lambda^*]$ such that \eqref{eq:D} possesses another positive solution $v_\lambda$ for any $\lambda < \lambda_*$.
\end{thm}
\begin{proof}
	We prove the existence of the first positive solution to \eqref{eq:D} for all $\lambda<0$ with sufficiently large $|\lambda|$ via the global minimization of the corresponding energy functional
	$$
	E_\lambda(u) = 
	\frac{1}{2} \int_\Omega |\nabla u|^2 \, dx 
	-
	\frac{1}{2} \int_\Omega |u^+|^{2} \, dx
	+
	\frac{1}{3} \int_\Omega |u^+|^{3} \, dx  
	+
	\frac{\lambda}{n+1} \int_\Omega |u^+|^{n+1} \, dx
	$$
	which is weakly lower-semicontinuous on $X$. 
	Note that $E_\lambda$ is coercive on the cone of nonnegative functions of $X$.
	Indeed, using the definition of $\lambda_1$ and the Sobolev embedding theorem, we get
	\begin{equation}\label{eq:caseVI1}
	E_\lambda(u) 
	\geq  
	\frac{\lambda_1-1}{2\lambda_1} \int_\Omega |\nabla u|^2 \, dx 
	+
	\frac{1}{3} \int_\Omega |u^+|^{3} \, dx  
	+
	\lambda C \left(\int_\Omega |u^+|^{3} \, dx \right)^\frac{n+1}{3},
	\end{equation}
	which easily implies that $E_\lambda(u) \to +\infty$ as $\|u\|_X \to +\infty$, provided $u \geq 0$ in $\Omega$.
	At the same time, fixing any $v \in X \setminus \{0\}$ such that $v \geq 0$ in $\Omega$, we see that $E_\lambda(v) < 0$ for all $\lambda<0$ with sufficiently large $|\lambda|$.
	Therefore, applying the direct minimization procedure, we obtain the existence of a critical point $u_\lambda$ of $E_\lambda$ for all such $\lambda<0$, and $u_\lambda$ satisfies $E_\lambda(u_\lambda)<0$. 
	Moreover, $u_\lambda > 0$ in $\Omega$ due to Lemma \ref{lem:smp}.

	Let us define the critical value
	$$
	\lambda^* = \sup\{\lambda \in \mathbb{R}:~ \eqref{eq:D} ~\text{has a solution}\}
	$$
	and obtain its upper bound by the same argument as in Lemma \ref{lem:IIIb}.
	Denoting by $u_\lambda$ a solution to \eqref{eq:D}, we multiply \eqref{eq:D} by $\varphi_1$ and obtain
	\begin{equation}\label{eq:lem:contr3x}
	\int_\Omega \left((\lambda_1-1) + u_\lambda + \lambda u_\lambda^{n-1} \right) u_\lambda \varphi_1 \,dx = 0.
	\end{equation}
	Clearly, we have $\lambda < 0$, since otherwise \eqref{eq:lem:contr3x} cannot be true.
	Considering the function $h_\lambda(s) = (\lambda_1-1) + s + \lambda s^{n-1}$ for $s>0$ and recalling that $1<n<2$, we see that the term $\lambda s^{n-1}$ is leading as $s \to 0$, and the term $s$ is leading as $s \to +\infty$. 
	We deduce that $h_\lambda''(s)>0$ for all $s>0$, and hence $h_\lambda(s)$ has exactly one critical point $s_\lambda$ which is the point of global minimum. 
	Performing direct calculations, we conclude that $h_\lambda(s_\lambda)$ increases with respect to $\lambda \in (-\infty,0)$.
	Thus, if $h_{\overline{\lambda}}(s_{\overline{\lambda}}) = 0$ for some $\overline{\lambda}$, then $h_\lambda(s_\lambda) > 0$ for all $\lambda > \overline{\lambda}$, and hence for such $\lambda$ we get a contradiction with \eqref{eq:lem:contr3x}. 
	Directly analyzing the function $h_\lambda(s)$, we see that
	$$
	\overline{\lambda} = -\frac{(\lambda_1-1)^{2-n}}{(n-1)^{n-1}(2-n)^{2-n}} < 0,
	$$
	which proves that $\lambda^* \leq \overline{\lambda} < 0$.
	
	Let us now introduce the critical value
	$$
	\lambda_* = \sup\{\lambda<0:~ \eqref{eq:D} ~\text{has a solution $u_\lambda$ such that $E_\lambda(u_\lambda)<0$}\}.
	$$
	We prove the existence of another positive solution $v_\lambda$ to \eqref{eq:D} for $\lambda < \lambda_*$, provided $N<6$, by means of the mountain pass theorem. 
	Note that since $N<6$, we have $3<2^*$, and hence $X=W_0^{1,2}(\Omega)$.
	Fixing any $\rho>0$ and considering any $u$ such that $\|\nabla u\|_2 = \rho$, we see from \eqref{eq:caseVI1} that
	$$
	E_\lambda(u) \geq \frac{\lambda_1-1}{2\lambda_1} \rho^2 + \lambda C \rho^{n+1},
	$$
	where we applied the Sobolev embedding theorem to the last term in \eqref{eq:caseVI1}.
	Since $n>1$, we see that for any sufficiently small $\rho>0$ there exists $\alpha>0$ such that $E_\lambda(u) > \alpha$ provided $\|\nabla u\|_2 = \rho$.
	On the other hand, by definition of $\lambda_*$, for any $\lambda<\lambda_*$ there exists $u_\lambda$ satisfying $E_\lambda(u_\lambda)<0$. 
	Therefore, taking sufficiently small $\rho>0$ such that $\rho < \|u_\lambda\|_2$, and noting that $E_\lambda$ satisfies the Palais-Smale condition on $W_0^{1,2}(\Omega)$ in view of its coercivity, we apply the mountain pass theorem and obtain a solution $v_\lambda$ to \eqref{eq:D} for any $\lambda<\lambda_*$, such that $E_\lambda(v_\lambda)>0$.
	This solution is positive in $\Omega$ due to Lemma \ref{lem:smp}.
\end{proof}

\begin{remark}
	Generalizations of the problem \eqref{eq:D} which allow indefinite weights have been studied, e.g., by \textsc{Alama \& Tarantello} in \cite{AT1} and \textsc{Shi \& Shivaji} in \cite{SS}.
	However, in these articles the placement of the parameter $\lambda$ is different than in our case. 
	While the right-hand side in \eqref{eq:D} under the case \ref{VII} reads as $u - \lambda u^n - u^2$ with $\lambda<0$, the results of \cite{AT1} would translate to the right-hand side of the form $\lambda u + u^n - u^2$ with $\lambda \in \mathbb{R}$, and the results of \cite{SS} to $\lambda(u + u^n - u^2)$ with $\lambda >0$.
	This fact prevents a direct application of the results from \cite{AT1,SS} to our problem.
\end{remark}

\begin{figure}[ht]
	\centering
	\includegraphics[width=0.6\linewidth]{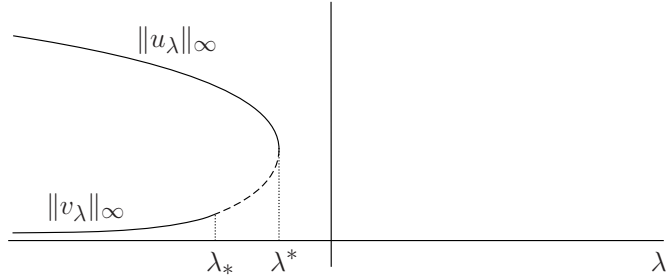}\\
	\caption{The branches of solutions to \eqref{eq:D} in the case $m=1 < n < \min\{2,2^*-1\}$ and $\lambda_1>1$.}
	\label{fig:caseVIIc}
\end{figure}

\begin{conjecture}
	We expect that the branches of solutions in the cases \ref{thm:VII2-1} and \ref{thm:VII2-2} of Theorem \ref{thm:VII1} have the same behaviour as depicted on Figure \ref{fig:caseIIb} and Figure \ref{fig:caseVIIb}, respectively. 
	Moreover, we anticipate that the branches of solutions obtained in Theorem \ref{thm:VIIb} behave as depicted on Figure \ref{fig:caseVIIc}.
	Notice that the question about the uniqueness in the case \ref{thm:VII2-2} of Theorem \ref{thm:VII1} remains open. 
\end{conjecture}

\subsection{Subcase $n=2$}
Under the assumption $n=m+1=2$, any solution to \eqref{eq:D} solves also
\begin{equation}\label{eq:Dmn=2}
\left\{
\begin{aligned}
-\Delta u &= u- (\lambda+1) u^2 &&\text{in}~ \Omega,\\
u &= 0 &&\text{on}~ \partial\Omega.
\end{aligned}
\right.
\end{equation}
The problem \eqref{eq:Dmn=2} is the problem with the logistic nonlinearity if $\lambda > -1$, \eqref{eq:Dmn=2} is the eigenvalue problem if $\lambda=-1$, and \eqref{eq:Dmn=2} is the superlinear problem if $\lambda<-1$. 
Thus, if $\lambda > -1$, then \eqref{eq:D} has a (unique) solution if and only if $\lambda_1 < 1$.
If $\lambda=-1$, then \eqref{eq:D} has infinitely many solutions (of the form $C\varphi_1$, $C>0$) if and only if $\lambda_1=1$.
If $\lambda < -1$ and $N < 6$ (or, equivalently, $3<2^*$), then \eqref{eq:D} has a solution if and only if $\lambda_1 > 1$.
Moreover, any solution to \eqref{eq:D} is positive in $\Omega$ due to Lemma \ref{lem:smp}.

\section{Case \ref{VIII}}\label{sec:VIII}
In this section, we let $m+1 = 2 < n$. 
First, we assume that $\lambda_1<1$.
As in Section \ref{sec:IVa}, the problem \eqref{eq:D} turns out to be a special case of the more general problem studied in \cite{DGU1,DGU2}. 
Applying the results of \cite{DGU2}, we deduce the existence of $\lambda^* \in (-\infty,0)$ such that \eqref{eq:D} possesses a positive solution $u_\lambda$, satisfying $E_\lambda(u_\lambda)<0$, for any $\lambda \in (\lambda^*,0)$, and \eqref{eq:D} has no solution for any $\lambda < \lambda^*$. 
Moreover, if $n<2^*-1$, then \eqref{eq:D} possesses another positive solution $v_\lambda$ for any $\lambda \in [\lambda_*,0)$, and $E_\lambda(v_\lambda)\leq 0$ for $\lambda=\lambda^*$.
See Figure \ref{fig:caseIV}, where the behaviour of the corresponding branches is depicted.

Let us complement these facts by considering the existence of solutions to \eqref{eq:D} for $\lambda \geq 0$ and estimating $\lambda^*$.
\begin{thm}\label{thm:caseVIII}
	Let  $m+1 = 2 < n$ and $\lambda_1<1$. 
	Then \eqref{eq:D} possesses a unique positive solution $u_\lambda$ for any $\lambda \geq 0$.
	Moreover, 
	$$
	\lambda^* > - \frac{(n-2)^{n-2}}{(n-1)^{n-1}(1-\lambda_1)^{n-2}}.
	$$
\end{thm}
\begin{proof}
	If $\lambda \geq 0$, then the existence and uniqueness of $u_\lambda$ are given by Theorems \ref{thm:existence2} and \ref{thm:uniq}. 
	To obtain the lower bound for $\lambda^*$, we argue in much the same way as in the proof of Lemma \ref{lem:IIIb}. 
	Assume that \eqref{eq:D} possesses a solution $u_\lambda$ for some $\lambda < 0$. 
	Multiplying the equation \eqref{eq:D} by $\varphi_1$ and integrating by parts, we get
	\begin{equation}\label{eq:lem:contr4}
	\int_\Omega \left((\lambda_1-1)  + u_\lambda + \lambda u_\lambda^{n-1} \right) u_\lambda \varphi_1 \,dx = 0.
	\end{equation}
	Let us analyze the function $h_\lambda(s) = (\lambda_1-1)  + s + \lambda s^{n-1}$ for $s>0$.
	Thanks to $n>2$, we see that $s$ is the leading term as $s \to 0$, and $\lambda s^{n-1}$ is the leading term as $s \to +\infty$. 
	Moreover, recalling that $\lambda<0$, we have $h_\lambda''(s)<0$ for all $s>0$.
	Therefore, $h_\lambda(s)$ has exactly one critical point $s_\lambda$ which is the point of global maximum. Arguing straightforwardly, we note that $h_\lambda(s_\lambda)$ increases with respect to $\lambda \in (-\infty,0)$.
	Thus, if $h_{\overline{\lambda}}(s_{\overline{\lambda}}) = 0$ for some $\overline{\lambda}<0$, then $h_\lambda(s_\lambda) < 0$ for all $\lambda < \overline{\lambda}$, and hence for such $\lambda$ we get a contradiction with \eqref{eq:lem:contr4}. 
	Directly investigating $h_\lambda(s)$, we find that
	$$
	\overline{\lambda} = - \frac{(n-2)^{n-2}}{(n-1)^{n-1}(1-\lambda_1)^{n-2}},
	$$
	which finishes the proof.
\end{proof}

\begin{remark}
	Another explicit estimate for $\lambda^*$ is given in \cite[Theorem 2.1]{DGU1}.
\end{remark}

\begin{figure}[ht]
	\centering
	\includegraphics[width=0.6\linewidth]{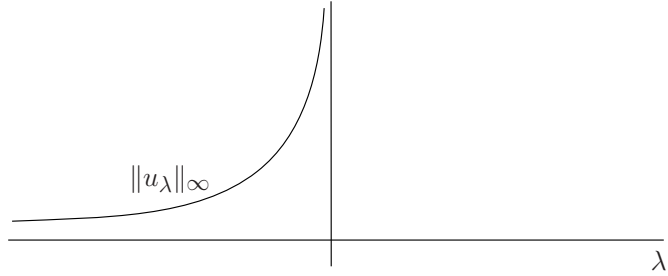}\\
	\caption{The branch of solutions to \eqref{eq:D} in the case $m+1 = 2 < n < 2^*-1$ and $\lambda_1 \geq 1$.}
	\label{fig:caseVIII}
\end{figure}

Let us now study the case $\lambda_1 \geq 1$.
\begin{thm}\label{thm:caseVIII2}
	Let  $m+1 = 2 < n < 2^*-1$ and $\lambda_1 \geq 1$. 
	Then \eqref{eq:D} possesses a positive solution $u_\lambda$ for any $\lambda < 0$, and \eqref{eq:D} has no solution for any $\lambda \geq 0$. 
\end{thm}
\begin{proof}
	We prove the existence of positive solutions for $\lambda < 0$ by the mountain pass theorem. 
	First, let us show that for any $\lambda<0$ there exist $\rho>0$ and $\alpha>0$ such that $E_\lambda(u) > \alpha$ provided $\|\nabla u\|_2 = \rho$.
	Let us fix any $\rho>0$ and consider an arbitrary $u$ such that $\|\nabla u\|_2 = \rho$.
	Applying H\"older's inequality and the Sobolev embedding theorem, we get
	$$
	E_\lambda(u) \geq \frac{\rho^2}{2} - \frac{1}{2}\int_\Omega |u^+|^2 \, dx + C \left(\int_\Omega |u^+|^2 \, dx\right)^\frac{3}{2} - \lambda C \rho^{n+1}.
	$$
	On the one hand, assuming that $\int_\Omega |u^+|^2 \, dx \leq \frac{1}{2}\rho^2$, we obtain
	$$
	E_\lambda(u) \geq \frac{\rho^2}{4} - \lambda C \rho^{n+1},
	$$
	which implies that $E_\lambda(u) \geq \alpha > 0$ for some $\alpha>0$ and all sufficiently small $\rho>0$. 
	On the other hand, assuming that $\int_\Omega |u^+|^2 \, dx > \frac{1}{2}\rho^2$ and recalling that $\lambda_1 \geq 1$, we apply the Sobolev embedding theorem to deduce that
	$$
	E_\lambda(u) 
	\geq 
	C \rho^3 - \lambda C \rho^{n+1}.
	$$
	Since $n+1 > 3$, we again see that $E_\lambda(u) \geq \alpha > 0$ for some $\alpha>0$ and all sufficiently small $\rho>0$, which establishes the claim.
	Second, taking any $w \in W_0^{1,2}(\Omega) \setminus \{0\}$ such that $w \geq 0$ in $\Omega$ and analyzing the fibers $\phi_w(t)$, we easily see that $E_\lambda(tw) < 0$ for all sufficiently large $t>0$. 
	Noting now that $E_\lambda(0)=0$ and $E_\lambda$ satisfies the Palais-Smale condition on $W_0^{1,2}(\Omega)$ by Lemma \ref{lem:PS}, and applying the mountain pass theorem, we obtain a solution $u_\lambda$ to \eqref{eq:D} for any $\lambda<0$. 
	This solution is positive in $\Omega$ due to Lemma \ref{lem:smp}.
	
	The nonexistence of solutions for $\lambda \geq 0$ can be obtained in much the same way as in Theorem \ref{thm:existence2}. 
\end{proof}

\begin{conjecture}
	In the case $\lambda_1 \geq 1$, we anticipate that the branches of solutions to \eqref{eq:D} have the behaviour as depicted on Figure \ref{fig:caseVIII}.
\end{conjecture}

\smallskip
\bigskip
\noindent
\textbf{Acknowledgments.}
V.\ Bobkov and P.\ Dr\'abek were supported by the grant 18-03253S of the Grant Agency of the Czech Republic. V. Bobkov was also supported by the project LO1506 of the Czech Ministry of Education, Youth and Sports.
The work of J.\ Hern\'andez was partially supported by Grant ref.\ MTM 2017-85449-P
of the Ministerio de Ciencia, Agencia Estatal de Investigaci\'on.

\end{document}